\documentclass[11pt]{article}
\usepackage{graphicx}
\usepackage{amsmath,amsthm,amssymb,enumerate}%, esint}
\usepackage{euscript,mathrsfs}
\usepackage{color}
\usepackage{dsfont}
\usepackage{url}
\usepackage[left=2cm,right=2cm,top=3.5cm,bottom=3.5cm]{geometry}
\usepackage{color}
\usepackage[framemethod=tikz]{mdframed}
\usepackage{soul}

\allowdisplaybreaks

\usepackage{soul}

\catcode`\@=11 \@addtoreset{equation}{section}

\catcode`\@=12

\newtheorem{Theorem}{Theorem}[section]
\newtheorem{Proposition}[Theorem]{Proposition}
\newtheorem{Lemma}[Theorem]{Lemma}
\newtheorem{Corollary}[Theorem]{Corollary}

\theoremstyle{definition}
\newtheorem{Definition}[Theorem]{Definition}

\newtheorem{Remark}[Theorem]{Remark}

\newcommand{\bTheorem}[1]{
	\begin{Theorem} \label{T#1} }
	\newcommand{\eT}{\end{Theorem}}

\newcommand{\bProposition}[1]{
	\begin{Proposition} \label{P#1}}
	\newcommand{\eP}{\end{Proposition}}

\newcommand{\bLemma}[1]{
	\begin{Lemma} \label{L#1} }
	\newcommand{\eL}{\end{Lemma}}

\newcommand{\bCorollary}[1]{
	\begin{Corollary} \label{C#1} }
	\newcommand{\eC}{\end{Corollary}}

\newcommand{\bRemark}[1]{
	\begin{Remark} \label{R#1} }
	\newcommand{\eR}{\end{Remark}}

\newcommand{\bDefinition}[1]{
	\begin{Definition} \label{D#1} }
	\newcommand{\eD}{\end{Definition}}

\newcommand{\Del}{\Delta_x}

\newcommand{\Ds}{\mathbb{D}_x}

\newcommand{\vuB}{\vc{u}_B}

\newcommand{\bFormula}[1]{
	\begin{equation} \label{#1}}
	\newcommand{\eF}{\end{equation}}

\newcommand{\Ov}[1]{\overline{#1}}

\newcommand{\aleq}{\stackrel{<}{\sim}}

\newcommand{\ageq}{\stackrel{>}{\sim}}

\newcommand{\vr}{\varrho}
\newcommand{\vre}{\vr_\ep}

\newcommand{\vte}{\vt_\ep}
\newcommand{\vue}{\vu_\ep}

\newcommand{\tvu}{{\wtilde \vu}}
\newcommand{\tvt}{\wtilde \vt}

\newcommand{\vt}{\vartheta}
\newcommand{\vu}{\vc{u}}

\newcommand{\vc}[1]{{\bf #1}}

\newcommand{\Div}{{\rm div}_x}
\newcommand{\Grad}{\nabla_x}

\newcommand{\dx}{\,{\rm d} {x}}

\newcommand{\dt}{\,{\rm d} t }

\newcommand{\intO}[1]{\int_{\Omega} #1 \ \dx}

\newcommand{\D}{{\rm d}}

\newcommand{\ep}{\varepsilon}

\newcommand{\vtB}{\vt_B}

\newcommand{\br}{ \nonumber \\ }

\def\softd{{\leavevmode\setbox1=\hbox{d}%
		\hbox to 1.05\wd1{d\kern-0.4ex{\char039}\hss}}}
\definecolor{Cgrey}{rgb}{0.85,0.85,0.85}
\definecolor{Cblue}{rgb}{0.50,0.85,0.85}
\definecolor{Cred}{rgb}{1,0,0}
\definecolor{fancy}{rgb}{0.10,0.85,0.10}
\definecolor{amaranth}{rgb}{0.9, 0.17, 0.31}

\newcommand\Cbox[2]{%
	\newbox\contentbox%
	\newbox\bkgdbox%
	\setbox\contentbox\hbox to \hsize{%
		\vtop{
			\kern\columnsep
			\hbox to \hsize{%
				\kern\columnsep%
				\advance\hsize by -2\columnsep%
				\setlength{\textwidth}{\hsize}%
				\vbox{
					\parskip=\baselineskip
					\parindent=0bp
					#2
				}%
				\kern\columnsep%
			}%
			\kern\columnsep%
		}%
	}%
	\setbox\bkgdbox\vbox{
		\color{#1}
		\hrule width  \wd\contentbox %
		height \ht\contentbox %
		depth  \dp\contentbox
		\color{black}
	}%
	\wd\bkgdbox=0bp%
	\vbox{\hbox to \hsize{\box\bkgdbox\box\contentbox}}%
	\vskip\baselineskip%
}

\mdfdefinestyle{MyFrame}{%
	linecolor=black,
	outerlinewidth=1pt,
	roundcorner=5pt,
	innertopmargin=\baselineskip,
	innerbottommargin=\baselineskip,
	innerrightmargin=10pt,
	innerleftmargin=10pt,
	backgroundcolor=white!20!white}

\newcommand{\wtilde}{\widetilde}

\allowdisplaybreaks

\begin{document}

%%%%%%%%%%%%%%%%%%%%%%%%%%%%%%%%

\title{\bf Local existence and conditional regularity for the Navier--Stokes--Fourier system driven by inhomogeneous boundary conditions}

\author{Anna Abbatiello\thanks{The work of A.A. is performed under the auspices of the Italian National Group of the Mathematical Physics (GNFM) of INdAM.} \and Danica Basari\' c \thanks{The work of D.B. was supported by the PRIN project 2022  ``Partial differential equations and related geometric-functional inequalities", financially supported by the EU, in the framework of the ``Next Generation EU initiative". The Department of Mathematics of Politecnico di Milano is supported by MUR ``Excellence Department 2023-2027".} \and Nilasis Chaudhuri \thanks{
The work of N.C. was supported by the ``Excellence Initiative Research University (IDUB)" program at the University of Warsaw.
} \and
Eduard Feireisl \thanks{The work of E.F. was partially supported by the
		Czech Sciences Foundation (GA\v CR), Grant Agreement
		24--11034S. The Institute of Mathematics of the Academy of Sciences of
		the Czech Republic is supported by RVO:67985840. E.F. is member of the Ne\v cas Center for Mathematical Modelling. } }

\date{}

\maketitle

\vspace{-5mm}
\medbreak

\centerline{University of Campania ``L.~Vanvitelli", Department of Mathematics and Physics}
\centerline{Viale A.~Lincoln 5, 81100 Caserta, Italy}

\centerline{anna.abbatiello@unicampania.it}

\medskip 

\centerline{Politecnico di Milano, Department of Mathematics}
\centerline{Via E. Bonardi 9, 20133 Milano, Italy}

\centerline{danica.basaric@polimi.it}

\medskip

\centerline{University of Warsaw, Faculty of Mathematics, Informatics and Mechanics}
\centerline{Stefana Banacha 2, 02-097 Warsaw, Poland} 
	
\centerline{nchaudhuri@mimuw.edu.pl}

\medskip

\centerline{Institute of Mathematics, Czech Academy of Sciences}
\centerline{\v Zitn\' a 25, 115 67 Praha 1, Czech Republic}

\centerline{feireisl@math.cas.cz} 

\begin{abstract}
	
We consider the Navier--Stokes--Fourier system with general inhomogeneous Dirichlet--Neumann boundary conditions. We propose a new approach 
to the local well--posedness problem based on conditional regularity estimates. By conditional regularity we mean that any strong solution
belonging to a suitable class remains regular as long as its amplitude remains bounded. The result holds for general Dirichlet--Neumann boundary conditions 
provided the material derivative of the velocity field vanishes on the boundary of the physical domain. As a corollary of this result we obtain:
\begin{itemize}
	\item Blow up criteria for strong solutions
	\item Local existence of strong solutions in the optimal  $L^p-L^q$ framework
	\item Alternative proof of the existing results on local well posedness 
\end{itemize}

\end{abstract}

\bigskip

{\small


\medbreak

\tableofcontents

}

\section{Introduction}
\label{i}

We propose a new approach to problems of well posedness of systems of equations arising in fluid dynamics based on {\it a priori} bounds resulting from 
conditional regularity estimates. The method consists in several steps: 

\begin{enumerate}
\item 
First, a local in time existence of smooth solutions is established in a class of very regular functions in the spirit of Nash's original paper \cite{NAS}.	
In fact, we use a more recent ``energy framework'' proposed by Matsumura and Nishida \cite{MANI}, later developed by Valli \cite{Vall2}, \cite{Vall1}, 
Valli and Zajaczkowski \cite{VAZA}, or Kagei and Kawashima \cite{KagKaw}. The energy method is based on a simple Hilbertian framework based on 
suitable scales of Sobolev spaces, where the desired level of regularity is achieved by successive differentiation of the field equations in the time variable. 
As a result, high regularity of the data is necessary accompanied by several levels of compatibility conditions to be satisfied. 

\item Probably optimal results can be achieved in the more recent framework of $L^p-L^q$ regularity spaces used in the seminal paper of Solonnikov \cite{SoloI}, and later revisited and developed by Danchin \cite{DanchCPDE}, \cite{Danch2010}, and Kotschote \cite{KOT6}. The $L^p-L^q$ class seems optimal both with respect to the required regularity of the data and the number of compatibility conditions.

\item We propose a new approach to the local existence in the $L^p-L^q$ setting based on suitable {\it a priori} bounds in terms of the data and 
a suitable \emph{control} functional evaluated in terms of suitable norms of the local solution. The essential ingredient of the method is 
weak continuity (compactness) of the control functional with respect to the norm of the function space in which the desired solution lives. A typical example 
of such a functional is the amplitude -- $L^\infty$-norm of the solution -- provided the underlying solution space is compactly embedded in the space of continuous function. With this type of estimates at hand, the local solution can be obtained in the desired class on a time interval determined by ``hitting'' time 
of the control functional. 

\item Besides the new local existence results, our approach yield a blow--up criterion in terms of the amplitude of the local solution in the 
spirit of the celebrated Nash's conjecture, cf. Nash \cite{Nash2}, and \cite{FeWeZh}.   

\end{enumerate}

The paper is organized as follows. In Section \ref{p}, we introduce the Navier--Stokes--Fourier system along with the relevant initial/boundary 
conditions and state our main results. Various concepts of strong solution are introduced in Section \ref{L}. Section \ref{CR} presents preliminary conditional regularity estimates. In Section \ref{LE} we establish the local existence result in the $L^p-L^q$ framework. The proof of conditional regularity and blow-up 
criterion are completed in Section \ref{R}.   

\section{Problem formulation and main results}
\label{p}

We introduce the initial/boundary value problem for the Navier--Stokes--Fourier system and state our main results.

\subsection{Navier--Stokes--Fourier system}

The time evolution of the mass density $\vr = \vr(t,x)$, the (absolute) temperature 
$\vt = \vt(t,x)$, and the velocity $\vu = \vu(t,x)$ of a viscous, compressible, and heat conducting fluid is governed by the \emph{Navier--Stokes--Fourier} (NSF) \emph{system} 
of equations:
\begin{align}
\partial_t \vr + \Div (\vr \vu ) &= 0, \label{p1} \\ 
\partial_t( \vr \vu) + \Div (\vr \vu \otimes \vu) + \Grad p (\vr, \vt) &= 
\Div \mathbb{S}(\Ds \vu) + \vr \Grad G, 
\label{p2}\\ 
\partial_t (\vr e(\vr, \vt)) + \Div(\vr e (\vr, \vt) \vu) + \Div \vc{q}(\Grad \vt) &= \mathbb{S} (\Ds \vu) : \Ds \vu - p(\vr, \vt) \Div \vu
\label{p3}
	\end{align}
satisfied for $(t,x) \in (0,T) \times \Omega$, $\Omega \subset R^3$, { where $\Ds \vu$ denotes the symmetric velocity gradient:
\begin{equation*}
	\Ds \vu \equiv \frac{1}{2} \left( \Grad \vu + \Grad^t \vu \right).
\end{equation*}
}
	
The pressure $p$ and the internal energy $e$ are related to the density and the 
temperature through the conventional \emph{Boyle--Mariotte} equation of state (EOS):
\begin{align} \label{p4}
p = p(\vr, \vt) = \vr \vt,\ e = e(\vt) = c_v \vt,\ c_v > 0.
\end{align}

The viscous stress $\mathbb{S}$ is given by \emph{Newton's rheological law}, 
while the heat flux $\vc{q}$ obeys \emph{Fourier's law}:
\begin{align} 
	\mathbb{S} &= \mathbb{S}(\Ds \vu) = \mu \Big( \Grad \vu + \Grad^t \vu - 
	\frac{2}{3} \Div \vu \mathbb{I} \Big) + \lambda \Div \vu, \label{p5} \\
	\vc{q} &= \vc{q}(\Grad \vt) = - \kappa \Grad \vt,    	
\label{p6}
\end{align}	
with constant transport coefficients $\mu > 0$, $\lambda \geq 0$, $\kappa > 0$.

The original state of the system is determined by the initial conditions 
\begin{equation} \label{p7}
	\vr(0, \cdot) = \vr_0,\ 
	\vt(0, \cdot) = \vt_0, \ \vu(0, \cdot) = \vu_0.
\end{equation}
In addition, we suppose the fluid occupies a bounded sufficiently smooth domain $\Omega \subset R^3$ and impose the Dirichlet boundary conditions for the velocity 
\begin{equation} \label{p8}
	\vu|_{\partial \Omega} = \vuB, \ \vuB \cdot \vc{n} = 0
\end{equation}
{ where $\vc{n}$ denotes the outer normal vector to $\partial\Omega$.}
Finally, we assume the boundary $\partial \Omega$ can be decomposed as
\begin{equation} \label{deco}
\partial \Omega = \Gamma_D \cup \Gamma_N,\ \Gamma_D \cap \Gamma_N = \emptyset, \ \Gamma_D, \ \Gamma_N \ \mbox{compact},
\end{equation}
where
\begin{equation} \label{p8a}  
	\vt|_{\Gamma_D} = \vtB,\ 
	\Grad \vt \cdot \vc{n}|_{\Gamma_N} = q_B.
\end{equation}
The boundary data $\vu_B = \vuB(x)$, $\vtB = \vtB(x)$, and $q_B = q_B(x)$ are independent of the time. The hypothesis $\vu_B$ is tangential to the boundary 
plays a crucial role in the analysis and could possibly be relaxed by imposing a more restrictive regularity criterion.

\subsection{Strong solutions in the $L^p-L^q$ framework}

Following the seminal paper by Solonnikov \cite{SoloI}, we look for solutions in the $L^p-L^q$ class. Specifically, 
\begin{align}
\vr &\in C([0,T]; W^{1,q}(\Omega)),\hspace{0.7cm} \partial_t \vr \in L^p(0,T; L^q(\Omega), \label{PP1} \\
\vt &\in L^p(0,T; W^{2,q}(\Omega)),\hspace{0.8cm} \partial_t \vt \in L^p(0,T; L^q(\Omega)), \label{PP2} \\
\vu &\in	L^p(0,T; W^{2,q}(\Omega;R^3)),\ \partial_t \vu \in L^p(0,T; L^q(\Omega;R^3)). \label{PP3}
	\end{align}  
Accordingly, 
\[
\vt \in C([0,T]; B^{2 (1 - \frac{1}{p})}_{q,p}(\Omega)),\quad 
\vu \in C([0,T]; B^{2 (1 - \frac{1}{p})}_{q,p}(\Omega; R^3)),  
\]
where the Besov space $B^{2 (1 - \frac{1}{p})}_{q,p}(\Omega))$ can be identified with the real interpolation space 
\[
B^{2 (1 - \frac{1}{p})}_{q,p}(\Omega) = \left[ L^q(\Omega); W^{2,q}(\Omega) \right]_{1 - \frac{1}{p},p},
\]
see Amann \cite[Section 5]{Amann1}, \cite[Chapter III, Theorem 4.10.2]{AMA}.

Our main result reads as follows. 
	\begin{Theorem}[{\bf NSF system - local solutions, blow up criterion}] \label{Tmain}
	
	Let $\Omega \subset R^3$ be a bounded domain, with the boundary of class $C^4$ admitting the decomposition \eqref{deco}. Let the exponents $p,q$,    
	\begin{equation} \label{LEa1} 
		3 < q < \infty, \ \frac{2q}{2q - 3} < p < \infty,
	\end{equation}
	be given.
	
	Let $G = G(x)$, $G \in W^{1,q}(\Omega)$ be a given potential. Let the initial data $(\vr_0, \vt_0, \vu_0)$ 
	and the boundary data $(\vtB, \vuB, q_B)$ belong to the class
	\begin{align}
		\vr_0 &\in W^{1,q}(\Omega), \hspace{0.7cm} \inf_{\Omega} \vr_0 \geq \underline{\vr} > 0, \label{PP4}\\ 
	    \vt_0 &\in B^{2 (1 - \frac{1}{p})}_{q,p}(\Omega), \hspace{0.25cm}\inf_{\Omega} \vt_0 \geq \underline{\vt} > 0, \label{PP5}\\
	    \vu_0 &\in B^{2 (1 - \frac{1}{p})}_{q,p}(\Omega;R^3), \label{PP6} \\
	    \vtB &\in W^{2 - \frac{1}{q}, q}(\Gamma_D),\ \inf_{\Gamma_D} \vtB \geq \underline{\vt} > 0, \label{PP7}\\ 
	    \vuB &\in W^{2 - \frac{1}{q}, q}(\partial \Omega; R^3),\ {\vuB \cdot \vc{n} = 0}, \label{PP8}\\ 
	    q_B &\in W^{1 - \frac{1}{q},q}(\Gamma_N),\ q_B \geq 0. \label{PP9}
	\end{align}
	In addition let the compatibility conditions 
	\begin{align} 
		\vu_0|_{\partial \Omega} &= \vuB, \label{L21} \\
		\vt_0|_{\Gamma_D} &= \vtB,\ 
		\Grad \vt_0 \cdot \vc{n}|_{\Gamma_N} = q_B \label{L22}
	\end{align}
	hold. Finally, suppose either $\Gamma_D \ne \emptyset$ or $q_B = 0$.

	Then the following holds.
	\begin{itemize}
	\item {\bf Local existence.} 
	There exists $T > 0$ such that the \textup{NSF} system \eqref{p1}--\eqref{p6}, 
	with the initial conditions \eqref{p7}, and the boundary conditions \eqref{p8}, \eqref{p8a}
	admits a unique solution $(\vr, \vt, \vu)$ in the class \eqref{PP1}--\eqref{PP3}. In addition, 
	\begin{equation} \label{PP10}
	\min_{[0,T] \times \Ov{\Omega}} \vr > 0,\ \min_{[0,T] \times \Ov{\Omega}} \vt > 0.
	\end{equation}
	
	\item{\bf Blow--up criterion.} 
	Suppose $G \in W^{2,q}(\Omega)$.
	Let 
	\[
	T_{\rm max} = \sup \left\{ T > 0 \ \Big|\ \mbox{the solution $(\vr, \vt, \vu)$ exists in}\ [0,T] \right\}.
	\]
	Then either $T_{\rm max} = \infty$ or 
	\begin{equation} \label{PP11}
	\limsup_{t \to (T_{\rm max})-} \left\| ( \vr, \vt, \vu ) (t, \cdot) \right\|_{C(\Ov{\Omega}; R^5)} = \infty.
		\end{equation}

	\end{itemize}
	
\end{Theorem}

\begin{Remark} \label{PR1}
The required regularity of $\partial \Omega$ is enforced by the method of construction and is not optimal. The result should hold for 
$\Omega$ of class $C^2$ or even $C^{1,1}$.
\end{Remark}	

\begin{Remark} \label{PR2}
Hypothesis \eqref{LEa1} allows to control the quadratic term $\mathbb{S}(\Ds \vu) : \Ds \vu$ in terms of the available parabolic 
estimates, cf. the proof of Proposition \ref{PCR2}. It is also necessary for the (compact) embedding 
\[
 B^{2 (1 - \frac{1}{p})}_{q,p}(\Omega) \hookrightarrow \hookrightarrow C(\Ov{\Omega}), 
\]	
cf. Amann \cite{Amann2000}. In particular, $T_{\rm max}$ is independent of $p,q$ as long as \eqref{LEa1} holds.
\end{Remark}	

\begin{Remark} \label{PR3}
While the positivity of the flux $q_B$ is used in the proof of minimum principle, see Section \ref{NM} below, the alternative 
$\Gamma_D \ne \emptyset$ or $q_B = 0$ is possibly only technical. In fact, it is used in Section \ref{CR} to extend the Neumann 
boundary condition on $\Gamma_N$ inside $\Omega$.	
	
	\end{Remark}

Theorem \ref{Tmain} extends the result of Kotschote \cite{KOT6} to the $L^p-L^q$ framework, $p \ne q$ in general. 
The blow--up criterion is new in the $L^p-L^q$ setting and the proof is based on a combination of three main ingredients:
\begin{itemize}
	\item Extension of the conditional regularity estimates obtained for more regular solutions 
	in \cite{BaFeMi} to a larger class introduced by Cho and Kim \cite{ChoKim1}.  
    \item New {\it a priori} estimates in the $L^p-L^q$ framework based on conditional regularity estimates. 
    \item Pasting the local solutions (and the corresponding estimates) in Cho--Kim class, with the solutions 
    in the $L^p-L^q$ class. 
\end{itemize}

The rest of the paper is devoted to the proof of Theorem \ref{Tmain} based on the above scenario. As a by product, we obtain 
local existence of solutions in the Cho and Kim class for inhomogeneous boundary data (Theorem \ref{TEL2}),  and new conditional regularity estimates in terms 
of the amplitude of strong solutions in the $L^p-L^q$ class (Theorem \ref{TR1}).

\section{New formulation and minimum principle}
\label{NM}

Using the specific EOS \eqref{p4}, we can rewrite \eqref{p1}--\eqref{p3} in the form 
\begin{align}
	\partial_t \vr + \vu \cdot \Grad \vr &= - \vr \Div \vu, \label{p1a}, \\ 
	\partial_t \vu + \vu \cdot \Grad \vu &= \frac{1}{\vr} \Div \mathbb{S}(\Ds \vu) 
	- \vt \Grad \log(\vr) - \Grad \vt + \Grad G, 
	\label{p2a}\\ 
	\partial_t \vt + \vu \cdot \Grad \vt  &= \frac{\kappa}{c_v \vr} \Del \vt + \frac{1}{c_v \vr}\mathbb{S}(\Ds \vu): \Ds \vu - \frac{1}{c_v} \vt \Div \vu. 
	\label{p3a}
\end{align}

\subsection{Minimum principle for the density}

The formulation \eqref{p1a}--\eqref{p3a} anticipates strict positivity of the density $\vr$. 
Formally, this property follows if the same restriction is imposed on the initial data, cf. \eqref{PP4}. If the velocity field is regular enough, the continuity equation yields 
\begin{equation} \label{p4a}
\vr(\tau,\cdot) \geq \min_\Omega \vr_0 \exp\left( - \int_0^\tau \| \Div \vu \|_{L^\infty(\Omega)} \dt \right),
\end{equation}
{for any $\tau \in [0,T]$.} The classes of strong solutions considered in this paper always allow for the Lagrangean formulation of the equation of continuity. Specifically, the velocity field always satisfies 
\begin{equation} \label{p3b}
\vu \in L^1(0,T; W^{1,\infty}(\Omega, R^3)).
\end{equation}
In particular, the density will be always bounded below away from zero as long as this is true initially. This justifies the lower bound on the density claimed in \eqref{PP10}.  

\subsection{Minimum principle for the temperature}

Similarly, we multiply equation \eqref{p3a} on $\exp \left( \frac{1}{c_v} \int_0^\tau 
\| {\Div \vu} \|_{L^\infty(\Omega)} \right)$ obtaining  
\begin{align}
	\vr \partial_t \Theta + \vr \vu \cdot \Grad \Theta - { \frac{\kappa}{c_v}} \Del \Theta &\geq 0 
	\ \mbox{in}\ (0,T) \times \Omega, \label{p4b}\\ \Theta(\tau, \cdot)|_{\Gamma_D} &= \exp \left( \frac{1}{c_v} \int_0^\tau 
	\| {\Div \vu} \|_{L^\infty(\Omega)} \right) \vtB
	,\br 
	\Grad \Theta (\tau, \cdot) \cdot \vc{n}|_{\Gamma_N}  &= \exp \left( \frac{1}{c_v} \int_0^\tau 
	\| {\Div \vu} \|_{L^\infty(\Omega)} \right) q_B, 
\nonumber
\end{align}
where we have set 
\[
\Theta = \vt \exp \left( \frac{1}{c_v} \int_0^\tau 
\| {\Div \vu} \|_{L^\infty(\Omega)} \right).
\]

Next, consider a function $F$, 
\[
F : R \to [0, \infty) \ \mbox{convex},  \ 
F(z) = 0 \ \mbox{for}\ z \geq 0,\ F'(z) < 0 \ \mbox{for}\ z < 0. 
\]
Multiplying	\eqref{p4b} on $F'(\Theta - \underline{\vt})$, where 
\[
\underline{\vt} = \min \left\{ \min_\Omega \vt_0, \min_{\Gamma_D} \vtB \right\},
\]
and integrating by parts, we get 
\[
\frac{\D }{\dt} \intO{ \vr F(\Theta - {\underline{\vt}}) } \leq 
{\frac{\kappa}{c_v}} \int_{\Gamma_N} q_B F'(\Theta- \underline{\vt} ) \D \sigma  
\]
Here, we have used the impermeability of the boundary $\vu \cdot \vc{n}|_{\partial \Omega} = \vuB \cdot \vc{n} = 0$. Thus if $q_B \geq 0$ we conclude, similarly to 
\eqref{p4a}, 
\begin{equation} \label{p5a}
	\vt(\tau,\cdot) \geq \min \left\{ \min_\Omega \vt_0,\min_{\Gamma_D} \vtB \right\} \exp\left( - \int_0^\tau \frac{1}{c_v}\| \Div \vu \|_{L^\infty(\Omega)} \dt \right), 
\end{equation}
{for any $\tau \in [0,T]$.} In view of hypotheses \eqref{PP5}, \eqref{PP7}, the temperature remains strictly positive as claimed in \eqref{PP10}.

\section{Classes of strong solutions}
\label{L}

Unfortunately, the proof of Theorem \ref{Tmain} cannot be performed in the 
function spaces of type $L^p-L^q$ in a direct manner. Instead, we have to introduce intermediate 
spaces of more regular strong solutions. 

\subsection{Hilbert space -- energy framework}	

First, we consider the Hilbert space setting introduced by 
Valli \cite{Vall2}, \cite{Vall1} and Valli, Zajaczkowski \cite{VAZA}. Specifically, the velocity and the temperature belong to the class:
\begin{align} 
\vt &\in L^2(0,T; W^{4,2}(\Omega)), \hspace{0.75cm} \partial^2_{tt} 
\vt \in L^2(0,T; L^2(\Omega)),	\label{L2} \\ 
\vu &\in L^2(0,T; W^{4,2}(\Omega;R^3)), \ \partial^2_{tt} \vu 
\in L^2(0,T; L^2(\Omega;R^3)). \label{L3}
	\end{align}
In particular, given $\vu$ satisfying \eqref{L3}, the time evolution of the density is uniquely determined by the equation of continuity \eqref{p1a}, 
\begin{equation} \label{L1}
	\vr \in C([0,T]; W^{3,2}(\Omega)), \ \partial_t \vr \in C([0,T]; W^{2,2}(\Omega))
\end{equation}	 	
as long as $\vr_0 \in W^{3,2}(\Omega)$.
	
Accordingly, we introduce the \emph{solution} space in the Valli class,  
\[
\mathcal{S}_V[0,T] = \left\{ (\vr, \vt, \vu) \Big| 
\ (\vr, \vt, \vu) \ \mbox{belong to the class}\ \eqref{L2}-\eqref{L1} \right\}.
\]	
The class $\mathcal{S}_V[0,T]$ endowed with the norm 
\begin{align} 
\| (\vr, \vt, \vu) \|_{\mathcal{S}_V[0,T]} &= \sup_{t \in [0,T]} \| \vr {(t, \cdot)} \|_{W^{3,2}(\Omega)} +
\sup_{t \in [0,T]} \| \partial_t \vr {(t, \cdot)} \|_{W^{2,2}(\Omega)} \br
&+
 \| (\vt, \vu) \|_{L^2(0,T; W^{4,2}(\Omega; R^4))}
+ \| \partial^2_{tt}(\vt, \vu)\|_{L^2(0,T; L^{2}(\Omega; R^4))} 
\label{L3a}
\end{align}
is a Banach space.
The data belong to the associated trace spaces:
\begin{align}
\vr_0 &\in W^{3,2}(\Omega),\ \vt_0 \in W^{3,2}(\Omega),\ 
\vu_0 \in W^{3,2}(\Omega; R^3), \label{L4} \\
\vt_B &\in W^{\frac{7}{2},2}(\Gamma_D)),\ 
\vu_B \in W^{\frac{7}{2},2}(\partial \Omega; R^3)), \label{L7} \\	
{q}_B &\in W^{\frac{5}{2},2}(\Gamma_N)) .
\label{L8a}
	\end{align}
Similarly to the above, we introduce the \emph{data} spaces in the Valli class:
\[
	\mathcal{D}_{V,I} = \left\{ (\vr_0, \vt_0, \vu_0) \ \Big|\
	(\vr_0, \vt_0, \vu_0) \ \mbox{belong to the class \eqref{L4}} \right\}, 
\]
with the norm 
\begin{equation} \label{L4a}
\| (\vr_0, \vt_0, \vu_0) \|_{\mathcal{D}_{ V,I}} = \| (\vr_0, \vt_0, \vu_0) \|_{W^{3,2}(\Omega; R^5)};
\end{equation}
and 
\[
\mathcal{D}_{V,B} = \left\{ (\vt_B, \vu_B, {q}_B) \ \Big|\
(\vt_B, \vu_B, {q_B}) \ \mbox{belong to the class \eqref{L7}, \eqref{L8a}} \right\}, 
\]
with the norm
\begin{equation} \label{L4b}
	\| (\vt_B, \vu_B, {q_B}) \|_{\mathcal{D}_{V,B}} = \| \vt_B \|_{W^{\frac{7}{2},2}(\Gamma_D)} + 
\| \vu_B \|_{W^{\frac{7}{2},2}(\partial \Omega; R^3)} + \| {q}_B \|_{W^{\frac{5}{2},2}(\Gamma_N)}.
\end{equation}
Note that 
\begin{equation} \label{L4ba}
\mathcal{S}_V [0,T] \hookrightarrow C([0,T]; \mathcal{D}_{V,I}).
	\end{equation}	
	
The following result was proved by Valli \cite[Theorem A]{Vall1}, see also \cite[Remark 3.3]{Vall1}.

\begin{Proposition}[{\bf Local existence - $L^2-$framework}] \label{PL1}
	Let $\Omega \subset R^3$ be a bounded domain of class $C^4$, and let $G = G(x)$, $G \in W^{3,2}(\Omega)$ be a given potential. Let the initial data $(\vr_0, \vt_0, \vu_0)$ satisfy \eqref{L4}, and let
	the boundary data {$( \vt_B, \vu_B, q_B)$ } belong to the class \eqref{L7}--\eqref{L8a}, where 
	\[
		\min_{\Omega} \vr_0 \geq \underline{\vr} > 0,\ 
		\min_{\Omega} \vt_0 \geq \underline{\vt} > 0,\ { \min_{\Gamma_D} \vt_B \geq \underline{\vt} >0}, \ q_B \geq 0.
	\]
	In addition let the compatibility conditions
	\begin{align} 
		&\quad \vt_0|_{\Gamma_D} = \vtB, \ \vu_0|_{\partial \Omega} = \vuB, \
		\Grad \vt_0 \cdot \vc{n}|_{\Gamma_N} = q_B, \label{L9} \\
		&\left( - 
		\vu_0 \cdot \Grad \vu_0 - \frac{1}{\vr_0} \Grad p(\vr_0, \vt_0) + \frac{1}{\vr_0} \Div \mathbb{S}(\Ds \vu_0) + \Grad G   \right) \Big|_{\partial \Omega} = 0,  \label{L11}\\ 
		&\left( - \vu_0 \cdot \Grad \vt_0 - \frac{1}{c_v \vr_0} \Div (\kappa \Grad \vt_0) + \frac{1}{c_v \vr_0} \mathbb{S} (\Ds \vu_0) : \Ds \vu_0  - \frac{1}{c_v} \vt_0 \Div \vu_0 \right)\Big|_{\Gamma_D} = 0,  
		\label{L12} \\ 
		&\Grad \left( - \vu_0 \cdot \Grad \vt_0 - \frac{1}{c_v \vr_0} \Div (\kappa \Grad \vt_0) + \frac{1}{c_v \vr_0} \mathbb{S} (\Ds \vu_0) : \Ds \vu_0  - \frac{1}{c_v} \vt_0 \Div \vu_0 \right) \cdot \vc{n} \Big|_{\Gamma_N} = 0 
		\label{L12b}
	\end{align}
hold. Finally, suppose either $\Gamma_D \ne \emptyset$ or $q_B = 0$.
	
	Then there exists $T > 0$ such that the NSF system \eqref{p1}--\eqref{p6}, 
	with the initial conditions \eqref{p7}, and the boundary conditions \eqref{p8}, \eqref{p8a} 
	admits a unique solution $(\vr, \vt, \vu)$ in the class \eqref{L2}--\eqref{L1}. Moreover, the density $\vr$ and the temperature $\vt$ are strictly positive, 
	specifically \eqref{p4a}, \eqref{p5a} hold for any $\tau \in [0,T]$.

\end{Proposition}

Given the data $(\vr_0, \vt_0, \vu_0; \vt_B, \vu_B, q_B) \in \mathcal{D}_{V,I} \times \mathcal{D}_{V,B}$, we define 	
\begin{align}
T_{{\rm max},V} &= T_{{\rm max},V} (\vr_0, \vt_0, \vu_0; \vt_B, \vu_B, q_B) \br &= 
\sup \left\{ T \geq 0 \ \Big|\ \mbox{the solution of the NSF system exists in}\ \mathcal{S}_V[0,T] \right\}.  
\nonumber
\end{align}	
By virtue of Proposition \ref{PL1}, $T_{{\rm max},V} > 0$ for any data 	$(\vr_0, \vt_0, \vu_0; \vt_B, \vu_B, q_B) \in \mathcal{D}_{V,I} \times \mathcal{D}_{V,B}$
satisfying the compatibility conditions \eqref{L9}--\eqref{L12b}. In general, $T_{\rm max,V}$ may differ from $T_{\rm max}$ introduced in Theorem \ref{Tmain}. As we shall see {below}, however, they coincide for sufficiently regular data.

One may consider data belonging to higher order Sobolev spaces and satisfying higher order compatibility conditions to obtain higher regularity 
of the corresponding solutions, cf. Kagei and Kawashima \cite{KagKaw}.

\subsection{Mixed setting}

In their quest for a framework that would accommodate vacuum in solutions, Cho and Kim \cite{ChoKim1} introduced a class of solutions of ``mixed'' type, specifically,
\begin{align} 
	\vt &\in L^2(0,T; W^{2,q}(\Omega)) \cap C([0,T]; W^{2,2}(\Omega)), \br \partial_t \vt &\in L^2(0,T; W^{1,2}(\Omega)) \cap C([0,T]; L^2(\Omega)),
	\label{MS1}	\\
	\vu &\in L^2(0,T; W^{2,q}(\Omega; R^3)) \cap C([0,T]; W^{2,2}(\Omega;R^3)), \br \partial_t \vu &\in L^2(0,T; W^{1,2}(\Omega;R^3)) \cap C([0,T]; L^2(\Omega; R^3)),
	\label{MS2}
\end{align} 
where $3 < q \leq 6$. The density $\varrho$, again uniquely determined via integration along characteristics, belongs to the class 
\begin{equation} \label{MS3}
	\vr \in C([0,T]; W^{1,q}(\Omega)),\ \partial_t \vr \in C([0,T]; L^q(\Omega)).	
\end{equation}

Analogously to the preceding section, we introduce the solution space 
\[
\mathcal{S}_{ChK} = \left\{ (\vr, \vt, \vu) \ \Big|\ (\vr, \vt, \vu) \ \mbox{belongs to the class \eqref{MS1}--\eqref{MS3}} \right\},
\]
with the norm 
\begin{align}
	\| (\vr, \vt, \vu) \|_{\mathcal{S}_{ChK}[0,T]} &= \sup_{t \in [0,T]} \| \vr { (t, \cdot)} \|_{W^{1,q}(\Omega)} + 
	\sup_{t \in [0,T]}\| \partial_t \vr{ (t, \cdot)} \|_{L^q(\Omega)}\br &+ \sup_{t \in [0,T]} \| (\vt, \vu){ (t, \cdot)} \|_{W^{2,2}(\Omega; R^4)} + 
	\| (\vt, \vu) \|_{L^2(0,T; W^{2,q}(\Omega; R^4)} \br 
	&+\sup_{t \in { [0,T]}} \| (\partial_t \vt, \partial_t \vu){ (t, \cdot)} \|_{L^2(\Omega; R^4))} + 
	\| (\partial_t \vt, \partial_t \vu) \|_{L^2(0,T; W^{1,2}(\Omega; R^4)}.	
	\label{MS4}
\end{align}

The associated data belong to the trace space
\begin{align}
	\vr_0 &\in W^{1,q}(\Omega), \ 
	\vt_0\in W^{2,2}(\Omega),\ 
	\vu_0 \in W^{2,2} (\Omega;R^3), \label{MS5} \\
	\vu_B &\in W^{2 - \frac{1}{q},q}(\partial \Omega), \
	\vt_B \in W^{2 - \frac{1}{q},q}(\Gamma_D)) , \label{MS6} \\ 
	q_B &\in W^{1 - \frac{1}{q},q}(\Gamma_N)) .
	\label{MS7}
\end{align}

{We introduce the data spaces in the Cho-Kim class
$$\mathcal{D}_{ChK, I}= \left\{ (\vr_0, \vt_0, \vu_0) \ \Big|\
	(\vr_0, \vt_0, \vu_0) \ \mbox{belong to the class \eqref{MS5}} \right\},$$
with the norm
\begin{equation} \label{MS8}
	\| (\vr_0, \vt_0, \vu_0) \|_{\mathcal{D}_{ChK,I}} = \| \vr_0 \|_{{ W^{1,q}}(\Omega)} + \| \vu_0 \|_{W^{2,2}(\Omega; R^3)} + 
	\| \vt_0 \|_{W^{2,2}(\Omega)},
\end{equation}
and
$$\mathcal{D}_{ChK,B} = \left\{ (\vt_B, \vu_B, {q}_B) \ \Big|\
(\vt_B, \vu_B, {q_B}) \ \mbox{belong to the class \eqref{MS6}, \eqref{MS7}} \right\}, 
$$
with the norm
}
\begin{equation} 
	\left\| (\vuB, \vtB, q_B ) \right\|_{\mathcal{D}_{ChK, B}} =  \| \vtB \|_{W^{2 - \frac{1}{q},q}(\Gamma_D)} 	
	+ { \| \vuB \|_{W^{2 - \frac{1}{q},q}(\Omega; R^3)} }+  \| q_B \|_{W^{1 - \frac{1}{q},q}(\Gamma_N)}. 
	\label{MS9}
\end{equation}

As a matter of fact, Cho and Kim \cite[Theorem 2.1]{ChoKim1} proved a local existence of solutions in the above class for the 
purely Dirichlet problem $\Gamma_N = \emptyset$, $\vu_B = 0$, and $\vt_B = 0$, where the last condition is rather unphysical. As a by product of our analysis, we extend this result to the inhomogeneous and physically relevant boundary conditions, see Theorem \ref{TEL2} below.

\subsection{$L^p-L^q$ setting}

Finally, we introduce the notation relevant to the $L^p-L^q$ framework used in Theorem \ref{Tmain}. First observe that the 
velocity field $\vu$ belongs to the space 
\[
L^p(0,T; W^{2,q}(\Omega; R^3)),\ W^{2,q} \hookrightarrow C^{1 + \beta} \ \mbox{as}\ q > 3. 
\]
In particular, the density may be recovered by the method of characteristics: 
\begin{equation} \label{L14b}
\vr \in C([0,T]; W^{1,q}(\Omega)),\ \partial_t \vr \in  L^p(0,T; L^q(\Omega))
\end{equation}	
as long as $\vr_0 \in W^{1,q}(\Omega)$, 
cf. Danchin \cite{Danch2010}.

Similarly to the preceding sections, we introduce the solution space 	
\[
\mathcal{S}_{p,q}[0,T] = \left\{ (\vr, \vt, \vu) \Big| 
\ (\vr, \vt, \vu) \ \mbox{belong to the class}\ \eqref{PP1}-\eqref{PP3} \right\}.
\]	
We define a norm on $S_{p,q}[0,T]$ as
\begin{align} 
	\| (\vr, \vt, \vu) \|_{\mathcal{S}_{p,q}[0,T]} &= \sup_{t \in [0,T]} \| \vr{ (t, \cdot)} \|_{W^{1,q}(\Omega)} + \| (\vt, \vu) \|_{L^p(0,T; W^{2,q}(\Omega; R^4))}
	+ \| \partial_t (\vr, \vt, \vu)\|_{L^p(0,T; L^q (\Omega; R^5))} \br 
	&+ \| (\vt, \vu)(0,\cdot) \|_{B^{2(1-\frac{1}{p})}_{q,p}(\Omega; R^4)}.
	\label{L3g}
\end{align}
Note carefully that the norm includes the norm of the initial data in the interpolation space. This fact will be exploited later in Section \ref{exipq}.	
	
The associated space of initial data reads
\[
\mathcal{D}_{p,q,I} = \left\{ (\vr_0, \vt_0, \vu_0) \Big|\ 	
\ (\vr_0, \vt_0, \vu_0) \ \mbox{belong to the regularity class \eqref{PP4}-\eqref{PP6}} \right\}, 
\]
with the norm 
\begin{equation} \label{L17}
\| (\vr_0, \vt_0, \vu_0) \|_{\mathcal{D}_{p,q,I}} = 
\| \vr_0 \|_{W^{1,q}(\Omega)} + \| { (\vt_0, \vu_0)} \|_{B^{2(1 - \frac{1}{p})}_{q,p}(\Omega; R^4)}.
	\end{equation}
The boundary data belong to the space
\[
\mathcal{D}_{p,q,B} = \left\{ (\vt_B, \vu_B, q_B) \Big|\ 	
\ ({\vt_B}, \vu_B, q_B) \ \mbox{belong to the regularity class \eqref{PP7}-\eqref{PP9}} \right\}, 
\]
with the norm
\[
\| (\vt_B, \vu_B, q_B) \|_{{\mathcal{D}_{p,q,B}}}= \| \vt_B \|_{W^{2 - \frac{1}{q},q}(\Gamma_D; R)} { +}
\| \vu_B \|_{W^{2 - \frac{1}{q},q}(\partial \Omega; R^3)} { +} \| {q}_B \|_{W^{1 - \frac{1}{q},q}(\Gamma_N)}.
\]
	
Recall that
\begin{equation}\label{L18}
\mathcal{S}_{p,q}[0,T] \hookrightarrow 
C([0,T]; \mathcal{D}_{p,q,I}).  
\end{equation}

Kotschote \cite[Theorem 2.1]{KOT6} showed the local existence for $p = q > 3$, with certain restrictions imposed {on} the boundary heat flux $q_B$. 
Theorem \ref{Tmain} extends his result to the general case $p \ne q$ and a larger class of heat fluxes imposed on $\Gamma_N$. In particular, we 
propose an alternative method of proof in the case $p = q$.

\section{Conditional regularity estimates}
\label{CR}

Having collected the necessary preliminary material, {we} deduce several conditional regularity estimates on the strong solution belonging to the 
regularity class used by Valli.  

First, we need a suitable extension of boundary data. Specifically, we set 
\begin{equation} \label{BD1}
	\Div \mathbb{S}(\Grad \widetilde{\vu}_B) = 0 \ \mbox{in}\ \Omega,\ 
	\widetilde{\vu}_B|_{\partial \Omega} = \vuB,	
\end{equation}	
and
\begin{align} 
	\Del \widetilde{\vt}_B = 0 \ \mbox{in}\ \Omega,\ 
	\widetilde{\vt}_B|_{\Gamma_D} &= \vtB,\ 
	\Grad \widetilde{\vt}_B \cdot \vc{n} |_{\Gamma_N} = q_B, \ \mbox{if}\ \Gamma_D \ne \emptyset, \br
	\widetilde{\vt}_B = \underline{\vt} &= \inf_{\Omega} \vt_0 \ \mbox{if}\ \Gamma_D = \emptyset.
	\label{BD2}
\end{align}	

Note that, by virtue of hypothesis \eqref{PP7}, ${\widetilde{\vtB}} > 0$ in $\Ov{\Omega}$. In what follows, we shall identify the boundary data with their extension, 
$\vtB \approx \widetilde{\vt}_B$, $\vuB \approx \widetilde{\vu}_B$.

\subsection{Conditional regularity in the mixed and Hilbertian setting}

The following result represents a variant of the conditional regularity estimates proved in \cite{BaFeMi}. 
Here and hereafter, the symbol $\Lambda$ denotes a generic 
positive and non--decreasing function, $\Lambda: [0, \infty) \to (0, \infty)$.

\begin{Proposition}[{\bf Conditional regularity - mixed and Hilbertian framework}] 
	\label{PCR1}
	Let the hypotheses of Proposition \ref{PL1} be satisfied. Let  $(\vr, \vt, \vu)$ be the unique strong solution of NSF system defined on some time interval $(0,T)$.
	
	Then there exists a non--decreasing (positive) function 
	$\Lambda : [0, \infty) \to (0, \infty)$ such that the following estimates hold:
\begin{itemize}	
\item	
		\begin{align} 
		&\left\| (\vr, \vt, \vu) \right\|_{\mathcal{S}_{ChK} [0,T]} + \sup_{t \in [0,T]} \left\| \left(\frac{1}{\vr}, \frac{1}{\vt} \right) { (t, \cdot)} \right\|_{C(\Ov{\Omega}; {R^2})} 
		\br &\leq 
		\Lambda \left(T + \left\| (\vr_0, \vt_0, \vu_0) \right\|_{\mathcal{D}_{ChK,I}} + \left\| (\vuB, \vtB, q_B) \right\|_{\mathcal{D}_{ChK,B}}  
		+ \left\| \left(\frac{1}{\vr_0}, \frac{1}{\vt_0} \right) \right\|_{C(\Ov{\Omega}; {R^2})} + \| G \|_{W^{2,2}(\Omega)}\right. \br  &\quad + \left.  \left\| (\vr, \vt, \vu ) \right\|_{C([0,T]\times \Ov{\Omega}; R^5)} 	\right),\ 3 < q \leq 6;
		\label{c5a}
	\end{align}
	
\item	 
	\begin{align} 
		 &\left\| (\vr, \vt, \vu) \right\|_{\mathcal{S}_V[0,T]} + \sup_{t \in [0,T]} \left\| \left(\frac{1}{\vr}, \frac{1}{\vt} \right) {(t, \cdot)} \right\|_{C(\Ov{\Omega}; {R^2})} 
		\br &\leq 
		\Lambda \left(T + \left\| (\vr_0, \vt_0, \vu_0) \right\|_{\mathcal{D}_{V,I}} + \left\| (\vuB, \vtB, q_B) \right\|_{\mathcal{D}_{V,B}}  
		+ \left\| \left(\frac{1}{\vr_0}, \frac{1}{\vt_0} \right) \right\|_{C(\Ov{\Omega}; {R^2})} + \| G \|_{W^{3,2}(\Omega)} \right. \br  &+ \left.  \left\| (\vr, \vt, \vu ) \right\|_{C([0,T] \times \Ov{\Omega}; R^5)} 	\right).
		\label{c5}
	\end{align}
\end{itemize}		
\end{Proposition}	

\begin{Remark} \label{RCR1}

The estimate \eqref{c5} was basically shown in \cite{BaFeMi}, see also \cite{AbBaCh}. The main novelty is the estimate \eqref{c5a} in the class of lower regularity that will be exploited later for proving local existence in the mixed setting.
	
\end{Remark}

\begin{Remark} \label{RRC1}
Roughly speaking, the conclusion of Proposition \ref{PCR1} can be seen as {\it a priori} estimates for the 
NSF system in terms of the data and the $L^\infty-$norm of the solution.	
	\end{Remark}

The rest of the section is devoted to the proof of Proposition \ref{PCR1}. We proceed in several steps.

 \subsubsection{Energy estimates}
 
 Any bounded solution $(\vr, \vt, \vu)$ of the NSF system satisfies the energy bounds  
 \begin{align} 
 	\int_0^T &\left( \| \vu \|^2_{W^{1,2}(\Omega; R^3)} + 
 	\| \vt \|_{W^{1,2}(\Omega)}^2 \right) \dt \br
 	&\leq \Lambda \Big( T + \|(\vr_0, \vt_0, \vu_0) \|_{\mathcal{D}_{ChK,I}} + \| (\vuB, \vtB, q_B) \|_{\mathcal{D}_{ChK, B}}
 	 + \| (\vr, \vt, \vu) \|_{L^\infty((0,T) \times \Omega)}	 \Big).
 	\label{pr4}
 \end{align}	

Indeed we multiply the momentum equation \eqref{p2} on $\vu - \vuB$ and integrate the resulting expression by parts obtaining
 \begin{align} 
 	\frac{\D}{\dt} &\frac{1}{2} \intO{ \vr |\vu - \vuB|^2 } + 
 	\intO{ \mathbb{S}(\Ds \vu - \Ds \vuB) : (\Ds \vu - \Ds \vuB) } \br
 	&= \intO{ {\vr \vu \cdot \Grad \vuB \cdot (\vu-\vuB)} } + \intO{ p(\vr, \vt) \Div (\vu - \vuB) } + 
 	\intO{ \vr \Grad G \cdot (\vu - \vu_B) },
 	\label{pr4b} 
 \end{align}	
 where, by virtue of Korn--Poincar\' e inequality, 
 \begin{equation} \label{pr4a}
 	\intO{ \mathbb{S}(\Ds \vu - \Ds \vuB) : (\Ds \vu - \Ds \vuB) } \ageq 
 	\| \vu - \vuB \|^2_{W^{1,2}_0(\Omega; R^3)}.	
 \end{equation}
 Thus the bound for $\vu$ claimed in \eqref{pr4} results from integrating 
 \eqref{pr4b} in time. 
 
 Similarly, we multiply the heat equation \eqref{p3} on $(\vt - \vtB)$ to deduce 
 \begin{align}
 	c_v \frac{\D }{\dt} &\intO{ \vr |\vt - \vtB|^2 } + \kappa \intO{ |\Grad \vt - 
 		\Grad \vtB |^2 } \br
 	&=c_v {\int_{\Omega} \vr (\vt-\vtB) \ \vu \cdot \Grad \vtB \dx } + \intO{ (\vt - \vtB) \mathbb{S}(\Ds \vu) : \Ds \vu  } -  \intO{(\vt - \vtB) p(\vr, \vt) \Div \vu  }.
 	\label{pr4c}
 \end{align} 	
 Combining \eqref{pr4c} with the previous estimate on $\Grad \vu$ we complete the proof of \eqref{pr4}.
 
 \subsubsection{Estimates on the material derivative}
 
 Similarly to \cite[Sections 4-6]{BaFeMi}, we adapt the arguments of Fang, Zi, and Zhang \cite[Section 3]{FaZiZh} to the present choice of boundary conditions. 
 
 We start by introducing the material derivative 
 \[
 D_t g = \partial_t g + \vu \cdot \Grad g.
 \]
 The momentum equation \eqref{p2} takes the form
 \begin{equation} \label{prr1}
 	\vr D_t \vu + \Grad p = \Div \mathbb{S}(\Ds \vu) + \vr \Grad G.
 \end{equation}
 Following \cite[Section 4.1]{BaFeMi} we multiply \eqref{prr1} on $D_t (\vu - \vuB)$ and perform by parts integration. This step is exactly the same as in 
 \cite[Section 4.1]{BaFeMi} yielding finally the inequality 
 \begin{align}
 	&\frac{1}{2} \frac{\D }{\dt} \intO{ \mathbb{S} (\Ds \vu) : \Ds \vu } - \frac{\D }{\dt} \intO{ p \Div \vu } + 
 	\frac{1}{2} \intO{ \vr |D_t \vu|^2 } \br 
 	&\leq \Lambda \left(T + \| (\vr_0, \vt_0, \vu_0) \|_{\mathcal{D}_{ChK,I}} + \| (\vuB, \vtB, q_B) \|_{\mathcal{D}_{ChK,B}} + 
 	\left\| \left(\frac{1}{\vr_0}, \frac{1}{\vt_0}\right) \right\|_{C(\Ov{\Omega}; {R^2})} + \| (\vr, \vt, \vu) \|_{L^\infty((0,T) \times \Omega; R^5)}  \right) \times \br 
 	&\times \left(1  + \intO{ \vr |D_t \vt| |\Grad \vu|    } + \intO{ |\Grad \vu |^3 } + \| G \|_{W^{2,2}(\Omega)} \right).
 	\label{g2a}
 \end{align}

 Next, following the original idea of Hoff \cite{HOF1} (cf. also \cite[Section 3, Lemma 3.3]{FaZiZh}), we compute the material derivative of the momentum equation 
 \eqref{p2}:
 \begin{align} 
 	\vr &D^2_t \vu + \Grad \partial_t p + \Div (\Grad p \otimes \vu)	 \br &= \mu \Big( \Del \partial_t \vu + \Div (\Del \vu \otimes \vu) \Big) + \left( \lambda + \frac{\mu}{3} \right) \Big(
 	\Grad \Div \partial_t \vu + \Div \left( (\Grad \Div \vu) \otimes \vu \right) \Big) + \vr \vu \cdot \Grad^2 G.
 	\label{g3}
 \end{align}
 Exactly as in \cite[Section 4.2]{BaFeMi}, we may take the scalar product of 
 \eqref{g3} with $D_t(\vu - \vuB)$, and, after a lengthy computation, obtaining 
 \begin{align}
 	&\frac{1}{2} \frac{\D}{\dt} \intO{ \vr |D_t (\vu - \vuB) |^2 }	+ \mu \intO{ |\Grad D_t (\vu - \vuB) |^2 } + \left( \lambda +\frac{\mu}{3}\right) \intO{ |\Div D_t ( \vu - \vuB )|^2 } \br 
 	&\leq \Lambda \left(T + \| (\vr_0, \vt_0, \vu_0) \|_{\mathcal{D}_{ChK,I}} + \| (\vuB, \vtB, q_B) \|_{\mathcal{D}_{ChK,B}} + 
 	\left\| \left(\frac{1}{\vr_0}, \frac{1}{\vt_0}\right) \right\|_{C(\Ov{\Omega})} + \| (\vr, \vt, \vu) \|_{L^\infty(0,T) \times \Omega}  \right) \times \br 
 	&\times \left( 1 + \intO{ \vr |D_t \vt |^2 } + \intO{ |\Grad \vu |^4 }
 	+ \intO{ \vr |D_t \vu |^2 } + \| G \|_{W^{2,2}(\Omega)} \right).
 	\label{g4}
 \end{align}

 We close the estimates following the technique of velocity decomposition 
 proposed in the seminal paper by Sun, Wang, and Zhang \cite{SuWaZh1}. To this end, we first multiply the heat equation \eqref{p3} on $\partial_t \vt$ and integrate over $\Omega$: 
 \begin{align} 
 	c_v \intO{ \vr |D_t \vt|^2 } & + \frac{\kappa}{2} \frac{\D }{\dt} \intO{ |\Grad \vt |^2 } - {\kappa} \frac{\D}{\dt} \int_{\Gamma_N} \vt q_B \ \D \sigma \br &= 
 	c_v \intO{ \vr D_t \vt \ \vu \cdot \Grad \vt } - \intO{ \vr \vt \ \Div \vu \ D_t \vt} + \intO{ \vr \vt \  \Div \vu\  \vu \cdot \Grad \vt } \br 
 	&+ \frac{\D }{\dt} \intO{ \vt \  \mathbb{S}(\Ds \vu) : \Grad \vu } \br	&- 
 	\mu \intO{ \vt \left( \Grad \vu + \Grad^t \vu - \frac{2}{3} \Div \vu \mathbb{I} \right): \left( \Grad \partial_t \vu + \Grad^t \partial_t \vu - \frac{2}{3} \Div \partial_t \vu \mathbb{I} \right) } \br
 	&- 2 {\lambda }\intO{ \vt \  \Div \vu \  \Div \partial_t \vu }. 	
 	\label{g12}
 \end{align}
 Here, the boundary integral corresponding to the boundary condition \eqref{p8a}
 reads 
 \begin{align}
 	\intO{ \Del \vt \partial_t \vt } &= 
 	\int_{\partial \Omega} \partial_t \vt \Grad \vt \cdot \vc{n}\ \D \sigma - 
 	\frac{1}{2} \frac{\D}{\dt} \intO{ |\Grad \vt |^2 }  \br &= 
 	\frac{\D }{\dt} \int_{\Gamma_N} \vt q_B \ \D \sigma   - 
 	\frac{1}{2} \frac{\D}{\dt} \intO{ |\Grad \vt |^2 } .
 	\nonumber
 \end{align}
 
 It turns out that we may close the estimates exactly as in 
 \cite[Sections 4-6]{BaFeMi}. 
 \begin{align} 
 	\sup_{t \in [0,T]} &\Big( \| \vu (t, \cdot) \|_{W^{1,2}(\Omega; R^3)} + \|\sqrt{\vr}  D_t \vu (t, \cdot) \|_{L^{2}(\Omega; R^3)} + \| \vt (t, \cdot) \|_{W^{1,2}(\Omega)} + \| \Grad \vu (t, \cdot) \|_{L^4(\Omega; R^{3\times 3})} \Big) \br 
 	&+ \int_0^T \intO{ |\Grad D_t \vu |^2 } \dt + 	\int_0^T \intO{ \vr |D_t \vt |^2 } \dt +	\int_0^T \| \vt \|^2_{W^{2,2}(\Omega)} \dt \br
  &\leq \Lambda \left( T + 	
 \| (\vr_0, \vt_0, \vu_0) \|_{\mathcal{D}_{ChK,I}} + \| (\vuB, \vtB, q_B) \|_{\mathcal{D}_{ChK,B}} + 
 \left\| \left(\frac{1}{\vr_0}, \frac{1}{\vt_0}\right) \right\|_{C(\Ov{\Omega})} \right. \br
 &+ \left. \left\| (\vr, \vt, \vu ) \right\|_{L^\infty((0,T) \times \Ov{\Omega}; R^5)} + \| G \|_{W^{2,2}(\Omega)} \right).
\nonumber
 \end{align}
 
{Combining the previous estimates as in \cite[Sections 5,6]{BaFeMi}, we additionally deduce that
\begin{equation*}
	\begin{aligned}
		&\sup_{t \in [0,T]} \left( \| \vr (t, \cdot) \|_{W^{1,q}(\Omega)} + \| \partial_t \vr (t,\cdot) \|_{L^q(\Omega)} \right) + \int_0^T \| \vu \|^2_{W^{2,q}(\Omega; R^3)} \dt \\
		&\leq \Lambda \left( T + 	
	\| (\vr_0, \vt_0, \vu_0) \|_{\mathcal{D}_{ChK,I}} + \| (\vuB, \vtB, q_B) \|_{\mathcal{D}_{ChK,B}}  + \left\| (\vr, \vt, \vu ) \right\|_{L^\infty((0,T) \times \Ov{\Omega}; R^5)} + \| G \|_{W^{2,2}(\Omega)} \right).
	\end{aligned}
\end{equation*}
Moreover, from the fact that 
\begin{equation*}
	\Div \mathbb{S}\big(\Ds (\vu- \vu_B)\big) = \vr D_t \vu + \vt \Grad \vr + \vr \Grad \vt - \vr \Grad G
\end{equation*}
and from the classical regularity theory for elliptic equations, we recover that 
\begin{equation*}
	\begin{aligned}
		& \sup_{t \in [0,T]} \left\| \vu(t, \cdot) \right\|_{W^{2,2}(\Omega; R^3)} \\
		& { \ \lesssim \sup_{t \in [0,T]} \left(\left\| \sqrt{\vr} D_t \vu (t, \cdot) \right\|_{L^2(\Omega; R^3)} + \| \Grad \vr (t, \cdot) \|_{L^2(\Omega; R^3)}  + \| \Grad \vt (t, \cdot)  \|_{L^2(\Omega; R^3)}+ \| \Grad G \|_{L^2(\Omega; R^3)}\right)  } \\
		&\leq \Lambda \left( T + 	
		\| (\vr_0, \vt_0, \vu_0) \|_{\mathcal{D}_{ChK,I}} + \| (\vuB, \vtB, q_B) \|_{\mathcal{D}_{ChK,B}} + 
		\left\| \left(\frac{1}{\vr_0}, \frac{1}{\vt_0}\right) \right\|_{C(\Ov{\Omega})} \right. \br
		&+ \left. \left\| (\vr, \vt, \vu ) \right\|_{L^\infty((0,T) \times \Ov{\Omega}; R^5)} + \| G \|_{W^{2,2}(\Omega)} \right).
	\end{aligned}
\end{equation*}
We can now use the minimum principle for the density \eqref{p4a} to deduce the desired estimates in terms of $\vr$ and $\vu$ appearing in \eqref{c5a}. To get the remaining estimates for $\vt$, we first have to take the time derivative of equation \eqref{p3}, multiply the resulting expression on $\partial_t \vt$ and integrate it over $\Omega$, obtaining
\begin{equation} \label{i1}
	\begin{split}
		c_v &\frac{\D }{\dt} \intO{ \vr |\partial_t \vt|^2 }  + \kappa \intO{ |\Grad \partial_t\vt |^2 } \\ 
		=& -c_v \intO{ \partial_t \vr \  |\partial_t \vt|^2 } -c_v \intO{ \partial_t \vr \, \partial_t \vt \,  \vu \cdot \Grad \vt}  - c_v \intO{ \vr \, \partial_t \vt \,  \partial_t \vu \cdot \Grad \vt} \\
		&+ 2 \intO{\partial_t \vt \, \mathbb{S}(\Ds \vu) : \Grad \partial_t\vu  } \\
		&- \intO{ \vt \, \partial_t \vr \, \partial_t \vt \, \Div \vu } - \intO{ \vr \, |\partial_t \vt|^2 \, \Div \vu } - \intO{ \vr \vt \, \partial_t \vt \Div \partial_t \vu}
	\end{split}
\end{equation}
By a Gr\"{o}nwall argument, using the previously achieved estimates, it is not difficult to deduce
\begin{equation*}
	\begin{aligned}
		\sup_{t \in [0,T]}  &\left\| \partial_t \vt(t, \cdot) \right\|_{L^{2}(\Omega)} + \int_0^T \| \partial_t \vt \|^2_{W^{1,2}(\Omega)} \dt \\
		&\leq \Lambda \left( T + 	
	\| (\vr_0, \vt_0, \vu_0) \|_{\mathcal{D}_{ChK,I}} + \| (\vuB, \vtB, q_B) \|_{\mathcal{D}_{ChK,B}} + 
	\left\| \left(\frac{1}{\vr_0}, \frac{1}{\vt_0}\right) \right\|_{C(\Ov{\Omega}; R^2)} \right. \br
		&+ \left. \left\| (\vr, \vt, \vu ) \right\|_{L^\infty((0,T) \times \Ov{\Omega}; R^5)} + \| G \|_{W^{2,2}(\Omega)} \right).
	\end{aligned}
\end{equation*}

{
Indeed, denoting with $I_k$, $k=1, \dots, 7$, the integrals appearing on the right-hand side of \eqref{i1}, we have 
\begin{align*}
	|I_1| &\lesssim \left\| \frac{1}{\vr} \right\|_{L^{\infty}} \|\partial_t \vr \|_{L^3} \| \sqrt{\vr} \partial_t \vt \|_{L^2} \|  \partial_t \vt \|_{L^6} \leq \varepsilon  \| \Grad \partial_t \vt \|_{L^2}^2 + C(\varepsilon) \left\| \frac{1}{\vr} \right\|_{L^{\infty}}^2 \|\partial_t \vr \|_{L^q}^2 \| \sqrt{\vr} \partial_t \vt \|_{L^2}^2 \\[0.2cm]
	|I_2| & \lesssim \|\partial_t \vr \|_{L^3} \|  \partial_t \vt \|_{L^6}  \| \Grad \vt \|_{L^{2}} \leq \varepsilon  \| \Grad \partial_t \vt \|_{L^2}^2 + C(\varepsilon) \ \|\partial_t \vr \|_{L^q}^2 \| \vt \|_{W^{1,2}}^2\\[0.2cm]
	|I_3| & \lesssim \| \partial_t \vt \|_{L^6} \| \partial_t \vu \|_{L^3} \| \Grad \vt \|_{L^2} \leq \varepsilon  \| \Grad \partial_t \vt \|_{L^2}^2  + C(\varepsilon)  \| \vt \|_{W^{1,2}}^2 \|  \Grad \partial_t \vu \|_{L^2}^2 \\[0.2cm]
	|I_4| & \lesssim \| \partial_t \vt \|_{L^6} \| \Grad \vu \|_{L^3}  \| \Grad \partial_t \vu\|_{L^2} \leq \varepsilon  \| \Grad \partial_t \vt \|_{L^2}^2  + C(\varepsilon) \| \vu \|_{W^{2,2}}^2 \| \Grad \partial_t \vu\|_{L^2}^2 \\[0.2cm]
	|I_5| & \lesssim \|\partial_t \vr \|_{L^3} \|  \partial_t \vt \|_{L^6}  \| \Grad \vu \|_{L^{2}} \leq \varepsilon  \| \Grad \partial_t \vt \|_{L^2}^2 + C(\varepsilon) \ \|\partial_t \vr \|_{L^q}^2 \| \vu \|_{W^{1,2}}^2\\[0.2cm]
	|I_6| &\lesssim \| \sqrt{\vr} \partial_t \vt \|_{L^2} \|  \partial_t \vt \|_{L^6} \| \Grad \vu \|_{L^3} \leq \varepsilon  \| \Grad \partial_t \vt \|_{L^2}^2  + C(\varepsilon) \| \vu \|_{W^{2,2}}^2 \| \sqrt{\vr} \partial_t \vt \|_{L^2}^2 \\[0.2cm]
	|I_7| & \lesssim \| \sqrt{\vr} \partial_t \vt \|_{L^2}  \| \Grad \partial_t \vu\|_{L^2} \lesssim   \| \Grad \partial_t \vu\|_{L^2}^2 + \| \sqrt{\vr} \partial_t \vt \|_{L^2}^2
\end{align*}
}
{Using once again the regularity theory for elliptic equations and the fact that the Sobolev embeddings $W^{1,2}(\Omega) \hookrightarrow L^q(\Omega)$, $W^{1,q}(\Omega) \hookrightarrow L^{\infty}(\Omega)$ for $3< q \leq 6$ lead to
	\begin{align*}
		\int_{0}^{T} \| \Grad \vu : \Grad \vu \|^2_{L^q (\Omega)}  \dt &\leq \int_{0}^{T} \| \Grad \vu \|^2_{L^q(\Omega; R^{3\times 3})} \| \Grad \vu \|_{L^{\infty}(\Omega; R^{3\times 3})}^2  \dt \\
		&\lesssim \int_{0}^{T} \| \Grad \vu \|^2_{W^{1, 2}(\Omega; R^{3\times 3})} \| \Grad \vu \|_{W^{1,q} (\Omega; R^{3\times 3})}^2  \dt \\
		&\lesssim \sup_{t \in [0, T]}  \| \vu(t, \cdot)\|^2_{W^{2,2} (\Omega; R^3)} \int_{0}^{T}  \|  \vu \|_{W^{2,q}(\Omega; R^3)}^2  \dt,
	\end{align*}
	we get the extra regularity for the temperature $\vt$:
\begin{equation*}
	\begin{aligned}
		\sup_{t \in [0,T]} &\left\| \vt(t, \cdot) \right\|_{W^{2,2}(\Omega)} + \int_0^T \| \vt \|^2_{W^{2,q}(\Omega)} \dt \\
		&\leq \Lambda \left( T + 	
		\| (\vr_0, \vt_0, \vu_0) \|_{\mathcal{D}_{ChK,I}} + \| (\vuB, \vtB, q_B) \|_{\mathcal{D}_{ChK,B}} + 
		\left\| \left(\frac{1}{\vr_0}, \frac{1}{\vt_0}\right) \right\|_{C(\Ov{\Omega}; R^2)} \right. \br
		&+ \left. \left\| (\vr, \vt, \vu ) \right\|_{L^\infty((0,T) \times \Ov{\Omega}; R^5)} + \| G \|_{W^{2,2}(\Omega)} \right).
	\end{aligned}
\end{equation*}
}
{Indeed, 
\begin{align*}
	\sup_{t \in [0,T]} \left\| \vt(t, \cdot) \right\|_{W^{2,2}} \lesssim \sup_{t \in [0,T]} \left(\left\| \partial_t \vt(t, \cdot) \right\|_{L^2} + \left\| \Grad  \vt(t, \cdot) \right\|_{L^2} + \| \Grad \vu (t, \cdot) \|_{L^4}^2 + \| \Grad \vu (t, \cdot) \|_{L^2}\right) \\
	\lesssim \sup_{t \in [0,T]} \left(\left\| \partial_t \vt(t, \cdot) \right\|_{L^2} + \left\|  \vt(t, \cdot) \right\|_{W^{1,2}} + \| \vu (t, \cdot) \|_{W^{2,2}}^2 + \| \vu (t, \cdot) \|_{W^{1,2}}\right),
\end{align*}
and
\begin{align*}
	\int_{0}^{T} \| \vt \|^2_{W^{2,q}} \dt &\lesssim  \int_{0}^{T} \left(\| \partial_t \vt \|^2_{L^q} + \| \Grad \vt \|^2_{L^q} + \| |\Grad \vu|^2 \|^2_{L^q} + \| \Grad \vu \|^2_{L^q}  \right) \dt \\
	&\lesssim  \int_{0}^{T} \left(\| \partial_t \vt \|^2_{W^{1,2}} + \|  \vt \|^2_{W^{1,q}} + \| \Grad \vu\|_{L^{\infty}}^2  \| \Grad \vu \|^2_{L^q} + \|  \vu \|^2_{W^{1,q}}  \right) \dt  \\
	&\lesssim  \int_{0}^{T} \left(\| \partial_t \vt \|^2_{W^{1,2}} + \|  \vt \|^2_{W^{2,2}} + \| \vu\|_{W^{2,q}}^2  \| \vu \|^2_{W^{2,2}} + \|  \vu \|^2_{W^{2,2}}  \right) \dt. 
\end{align*}
{Putting all the estimates together and using the minimum principle for the temperature \eqref{p5a}, we end up to}
 \begin{align} 
 	&\sup_{t \in [0,T]} \left( \| \vr (t, \cdot) \|_{W^{1,{ q}}(\Omega)} + \| \partial_t \vr (t,\cdot) \|_{L^{ q}(\Omega)} \right)	\br 
 	&+ \sup_{t \in [0,T]} \left( \left\| \vu(t, \cdot) \right\|_{W^{2,2}(\Omega; R^3)}
 	+ \left\| \partial_t \vu(t, \cdot) \right\|_{L^{2}(\Omega; R^3)}
 	\right) + \sup_{t \in [0,T]} \left( \left\| \vt(t, \cdot) \right\|_{W^{2,2}(\Omega)}
 	+ \left\| \partial_t \vt(t, \cdot) \right\|_{L^{2}(\Omega)}
 	\right)
 	\br  
 	&+ \int_0^T \| \vu \|^2_{W^{2,{ q}}(\Omega; R^3)} \dt + \int_0^T \| \partial_t \vu \|^2_{W^{1,2}(\Omega; R^3)} \dt + \int_0^T \| \vt \|^2_{W^{2,{ q}}(\Omega)} \dt + \int_0^T \| \partial_t \vt \|^2_{W^{1,2}(\Omega)} \dt \br 
 	&+ \sup_{t \in [0,T]} \left\| \left( \frac{1}{\vr} , \frac{1}{\vt} \right) {(t, \cdot)}\right\|_{C(\Ov{\Omega}; R^2)} \br 
 	&\leq \Lambda \left( T + 	
 	 \| (\vr_0, \vt_0, \vu_0) \|_{\mathcal{D}_{ChK,I}} + \| (\vuB, \vtB, q_B) \|_{\mathcal{D}_{ChK,B}} + 
 	\left\| \left(\frac{1}{\vr_0}, \frac{1}{\vt_0}\right) \right\|_{C(\Ov{\Omega}; R^2)} \right. \br
 	&+ \left. \left\| (\vr, \vt, \vu ) \right\|_{L^\infty((0,T) \times \Ov{\Omega}; R^5)} + \| G \|_{W^{2,2}(\Omega)} \right).
 	\label{pr5}
 \end{align}
 This is nothing other than \eqref{c5a}.
 
 The proof of \eqref{c5} can be obtained exactly as in \cite{BaFeMi} passing to higher order derivatives. 
 We have therefore completed the proof of Proposition \ref{PCR1}.

\subsection{Conditional regularity estimates in the $L^p-L^q$ setting}

The reader will have noticed that the conditional regularity estimates in terms of the norm $\| (\vr, \vt, \vu) \|_{C([0,T] \times \Ov{\Omega}; R^5)}$ require 
differentiating the field equations in time. Unfortunately, this cannot be performed in a direct manner within the $L^p-L^q$ framework. Instead, we report the following auxiliary result.

\begin{Proposition}[{\bf Conditional regularity {estimates} - $L^p-L^q$ framework}] 
	\label{PCR2}
	Let the hypotheses of Proposition \ref{PL1} be satisfied. Let  $(\vr, \vt, \vu)$ be the unique solution of \textup{NSF} system defined on some time interval $(0,T)$.
	Let
	\begin{equation} \label{hhyp}
	3 < q < \infty,\ \frac{2q}{2q - 3} < p < \infty.
	\end{equation}
	
	Then there exists a non--decreasing (positive) function 
	$\Lambda : [0, \infty) \to (0, \infty)$ such that 
	the following estimate
	\begin{align} 
		&\left\| (\vr, \vt, \vu) \right\|_{\mathcal{S}_{p,q} [0, \tau]} + \sup_{t \in [0, \tau]} \left\| \left(\frac{1}{\vr}, \frac{1}{\vt} \right) { (t, \cdot)}\right\|_{C(\Ov{\Omega}; R^2)} 
		\br &\leq 
		\Lambda \left(\left\| (\vr_0, \vt_0, \vu_0) \right\|_{\mathcal{D}_{p,q,I}} + \left\| (\vuB, \vtB, q_B) \right\|_{\mathcal{D}_{p,q,B}}  
		+ \left\| \left(\frac{1}{\vr_0}, \frac{1}{\vt_0} \right) \right\|_{C(\Ov{\Omega}; R^2)}  + \| G \|_{W^{1,q}(\Omega)}\right. \br  &\quad + \left.  \left\| (\vr, \vt, \vu) \right\|_{C([0,\tau] \times \Ov{\Omega}; R^5)} + \int_0^{\tau} \| (\vt, \vu) \|^p_{W^{1, \infty}(\Omega; R^4)} \dt 	\right) 
		\label{CR1}
	\end{align}
holds for any $0 \leq \tau \leq T_{\rm eff} \wedge T$, where 
\begin{align} \label{CR2}
		T_{\rm eff} &\geq \Lambda^{-1} \left( \left\| (\vr_0, \vt_0, \vu_0) \right\|_{\mathcal{D}_{p,q,I}} + \left\| \frac{1}{\vr_0} \right\|_{C(\Ov{\Omega})} + \| G \|_{W^{1,q}(\Omega)} \right. \br &\left. + \sup_{t \in [0,T]} \| \vt \|_{C(\Ov{\Omega})} + \| \vu \|_{L^p (0,T; W^{1, \infty}(\Omega; R^3))} \right) > 0.
	\end{align}

\end{Proposition}	

\begin{Remark} \label{RCR2}
Obviously, the conclusion of Proposition \ref{PCR2} is ``weaker'' than that of Proposition \ref{PCR1} as the right--hand side of 
\eqref{CR1} depends on the norm 
$ \| (\vt, \vu) \|_{L^p (0,T; W^{1, \infty}(\Omega; R^3))}$, and 
the result is ``local'' in view of \eqref{CR2}. We will remove this drawback in the last part of the paper.
	\end{Remark}
	
\begin{proof}

First,  we rewrite the system \eqref{p1a}--\eqref{p3a} in the form 	
\begin{align}
	\partial_t b(\vr) + \vu \cdot \Grad b(\vr) &= - b'(\vr) \vr \Div \vu, \label{m4}, \\ 
	\partial_t (\vu - \vu_B)  &- \frac{1}{\vr_0} \Div \mathbb{S}(\Ds (\vu - \vu_B))  = \left( \frac{1}{\vr} - \frac{1}{\vr_0} \right) \Div \mathbb{S}(\Ds (\vu - \vu_B)) \br  &- \vu \cdot \Grad \vu
	- \vt \Grad \log(\vr) - \Grad \vt  + \Grad G, 
	\label{m5}\\ 
	\partial_t (\vt -\vt_B) - \frac{\kappa}{c_v \vr_0} \Del (\vt - \vt_B)   &= \left( \frac{\kappa}{c_v \vr} - \frac{\kappa}{c_v \vr_0} \right) \Del (\vt - \vt_B) - \vu \cdot \Grad \vt\br  &+ \frac{1}{c_v \vr}\mathbb{S}(\Ds \vu): \Ds \vu - \frac{1}{c_v} \vt \Div \vu, 
	\label{m6}
\end{align}
where $b$ in \eqref{m4} is a continuously differentiable function, and $\vu_B$, $\vt_B$ are the extensions of the boundary data introduced in \eqref{BD1}, \eqref{BD2}. In particular,
the functions $(\vu - \vu_B)$, $(\vt - \vt_B)$ satisfy the homogeneous boundary conditions.

 At this stage, we use the $L^p-L^q$ theory for parabolic equations, see 
 Denk, Hieber, and Pr\"uss \cite[Theorem 2.3] {DEHIEPR}, to deduce the estimates
 \begin{align}
 	&\int_0^T \| \partial_t (\vu - \vu_B) \|_{L^q(\Omega; R^3)}^p \dt + 
 	\int_0^T \left\| \frac{1}{\vr_0} \Div \mathbb{S}(\Ds (\vu - \vu_B) ) \right\|_{L^q(\Omega; R^3)}^p \dt \br 
 	&\leq  \Lambda \left( \left\| \frac{1}{\vr_0} \right\|_{C(\Ov{\Omega})} + \| \vr_0 \|_{W^{1,q}(\Omega)} \right) \left[  \int_0^T  \left\| \left( \frac{1}{\vr} - \frac{1}{\vr_0} \right) \Div \mathbb{S}(\Ds (\vu - \vu_B)) \right\|_{L^q(\Omega; R^3)}^p \dt \right. \br &+ 
 	\int_0^T \| \vu \cdot \Grad \vu \|_{L^q(\Omega; R^3)}^p  \dt  
 	+ \int_0^T  \| \Grad \vt \|_{L^q(\Omega; R^3)}^p  \dt  + \int_0^T \| \Grad G \|_{L^q(\Omega; R^3)}^p \br
 	&+ \int_0^T \| \vt \Grad \log(\vr) \|_{L^q(\Omega; R^3)}^p  \dt + \left\| \vu_0 - \vu_B \right\|^p_{B^{2(1- \frac{1}{p})}_{q,p}; R^3)}  \Big],
 	\label{MAX1}
 \end{align}
 and
 \begin{align}
 	\int_0^T &\| \partial_t (\vt - \vt_B) \|_{L^q(\Omega)}^p \dt + 
 	\int_0^T \left\| \frac{\kappa}{c_v \vr_0} \Del (\vt - \vt_B) ) \right\|_{L^q(\Omega)}^p \dt \br 
 	&\leq \Lambda \left( \left\| \frac{1}{\vr_0} \right\|_{C(\Ov{\Omega})} + \| \vr_0 \|_{W^{1,q}(\Omega)} \right)  \left[ \int_0^T  \left\| \left( \frac{1}{\vr} - \frac{1}{\vr_0} \right) \Del (\vt - \vt_B)) \right\|_{L^q(\Omega)}^p \dt \right. \br 
 	& + 
 	\int_0^T \left( \| \vu \cdot \Grad \vt) \|_{L^q(\Omega)}^p     \right) \dt+ \int_0^T  \left\| \frac{1}{\vr} \mathbb{S}(\Ds \vu) : \Ds \vu \right\|_{L^q(\Omega; R^3)}^p  \dt  + \int_0^T \left\| \frac{1}{c_v} \vt \Div \vu  \right\|_{L^q(\Omega)}^p   \dt\br
 	&+ \int_0^T \left\|\vt_0 - \vtB \right\|_{B^{2(1 - \frac{1}{p})}_{q,p} }\Big].
 	\label{MAX2}
 \end{align}
 Note that the constant $\Lambda$ depends, in general, on the modulus of continuity of $\vr_0$ and the ellipticity constants of the associated linear operator, 
 see e.g. Krylov \cite{Krylov}, Schlag \cite{Schlag}, or Solonnikov \cite{SoloI}. 
	
Moreover, using \eqref{BD1}, \eqref{BD2} we may rewrite the above inequalities in the form  
\begin{align}
	&\int_0^T \left[ \| \partial_t \vu  \|_{L^q(\Omega; R^3)}^p  + 
	\left\| \vu \right\|_{W^{2,q}(\Omega; R^3)}^p \right] \dt \br 
	&\aleq  \Lambda \left( \left\| \frac{1}{\vr_0} \right\|_{C(\Ov{\Omega})} + \| \vr_0 \|_{W^{1,q}(\Omega)} \right)  \left[  \int_0^T  \left\| \left( \frac{1}{\vr} - \frac{1}{\vr_0} \right) \Div \mathbb{S}(\Ds \vu) \right\|_{L^q(\Omega; R^3)}^p \dt \right. \br&+ 
	\int_0^T \| \vu \cdot \Grad \vu \|_{L^q(\Omega; R^3)}^p  \dt   
	+ \int_0^T  \| \Grad \vt \|_{L^q(\Omega; R^3)}^p  \dt  + \int_0^T \| \Grad G \|_{L^q(\Omega; R^3)}^p\br
	& + \left. \int_0^T \| \vt \Grad \log(\vr) \|_{L^q(\Omega; R^3)}^p  \dt+  \| \vu_0 \|_{B^{2(1 - \frac{1}{p})}_{q,p}(\Omega; R^3)}^p + (1 + T)\| \vuB \|^p_{W^{2- \frac{1}{q},q}({ \partial \Omega}; R^3)} \right],
	\label{m7}
\end{align}
and
\begin{align}
	&\int_0^T \left[ \| \partial_t \vt  \|_{L^q(\Omega)}^p \dt + 
	\int_0^T \left\| \vt \right\|_{W^{2,q} (\Omega)}^p \right] \dt \br 
	&\aleq \Lambda \left( \left\| \frac{1}{\vr_0} \right\|_{C(\Ov{\Omega})} + \| \vr_0 \|_{W^{1,q}(\Omega)} \right) \left[ \int_0^T  \left\| \left( \frac{1}{\vr} - \frac{1}{\vr_0} \right) \Del \vt \right\|_{L^q(\Omega)}^p \dt + 
	\int_0^T \left( \| \vu \cdot \Grad \vt) \|_{L^q(\Omega)}^p     \right) \dt \right. \br 
	&+ \int_0^T  \left\| \frac{1}{\vr} \mathbb{S}(\Ds \vu) : \Ds \vu \right\|_{L^q(\Omega)}^p  \dt  + \int_0^T \left\|  \vt \Div \vu  \right\|_{L^q(\Omega)}^p   \dt\br
	&+  \| \vt_0 \|_{B^{2(1 - \frac{1}{p})}_{q,p}(\Omega)} +  (1 + T)\| \vtB \|^2_{W^{2- \frac{1}{q},q}({ \Gamma_D})}  \Big].
	\label{m8}	
\end{align}

Next, we use the estimate on the density gradient proved by Kotschote \cite[Lemma 4.2 and comments]{KOT6}:
\begin{equation} \label{b7}
\sup_{\tau \in (0,T)} \| \Grad \vr (\tau, \cdot) \|_{L^q(\Omega; R^3)} 
\aleq \exp \left( 2 \int_0^T \| \Div \vu \|_{L^\infty(\Omega)} \dt \right)
\Big( \| \Grad \vr_0 \|_{L^q(\Omega{;R^3})} + \int_0^T \| \vu \|_{W^{2,q}(\Omega; R^3)} \Big).
\end{equation}
In particular, 
\begin{align}
&\int_0^T \| \vt \Grad \log(\vr) \|^p_{L^q(\Omega)} \dt \leq T \left[ \sup_{t \in (0,T)} \| \vt \|_{L^\infty(\Omega)} 
\sup_{t \in (0,T)} \left\| \frac{1}{\vr} \right\|_{L^\infty(\Omega)}  \sup_{t \in (0,T)} \| \Grad \vr  \|^p_{L^q(\Omega; R^3)} \right] \br 
&\aleq T \left[ \sup_{t \in (0,T)} \| \vt \|_{L^\infty(\Omega)} 
\sup_{t \in (0,T)} \left\| \frac{1}{\vr} \right\|_{L^\infty(\Omega)} \exp \left( 2 p \int_0^T \| \Div \vu \|_{L^\infty(\Omega)} \dt \right) \times \right. 
\br &\left.
\times \Big( \| \Grad \vr_0 \|^p_{L^q(\Omega {;R^3})} + T^{p-1} \int_0^T \| \vu \|^p_{W^{2,q}(\Omega; R^3)} \Big)  \right].
\label{b7a}	
	\end{align}
	
Finally, computing $\frac{1}{\vr}$ along characteristics we have 
\begin{align} 
	\frac{1}{\vr (t, X(t,x))} &= \frac{1}{\vr_0(x)} \exp \left( \int_0^t \Div \vu(s, X(s,x) \D s \right), 
\br
X'(t,x) &= \vu(t, X(t,x)),\ X(0,x) = x.
\nonumber
\end{align}
Consequently, 
\begin{align}
	\frac{1}{\vr (t, X(t,x))} - 	\frac{1}{\vr_0 (X(t,x))} = 
\frac{1}{\vr_0(x)} \left[ \exp \left( \int_0^t \Div \vu(s, X(s,x) \D s \right) - 1 \right] + 
\frac{1}{\vr_0(x)} -	\frac{1}{\vr_0 (X(t,x))}.
\nonumber	
	\end{align}
As $\vr_0 \in W^{1,q}(\Omega)$, $q > 3$, $\vr_0$ bounded below, the function $\frac{1}{\vr_0}$ is H\" older continuous. In particular,
\[ 
\left| \frac{1}{\vr_0(x)} -	\frac{1}{\vr_0 (X(t,x))} \right| \aleq \left| x - X(t,x) \right|^\beta = 
 \left| \int_0^t \vu (x, X(s,x)) \D s \right|^\beta
\]
{for some $0<\beta <1$. }

Thus for any $\delta > 0$, we may find $T_{\rm eff} > 0$ small enough in terms of  $\int_0^{T} \| \vu \|^p_{W^{1,\infty}(\Omega; R^3)}$ and the norm of the data so 
that 
\[  
\sup_{\tau \in (0,T_{\rm eff})} \left\| \left( \frac{1}{\vr(\tau, \cdot) } - \frac{1}{\vr_0} \right) \right\|_{L^\infty(\Omega)} < \delta.
\]
In particular, the integrals 
\[
\int_0^T  \left\| \left( \frac{1}{\vr} - \frac{1}{\vr_0} \right) \Div \mathbb{S}(\Ds \vu) \right\|_{L^q(\Omega; R^3)}^p \ \mbox{and}\ \int_0^T  \left\| \left( \frac{1}{\vr} - \frac{1}{\vr_0} \right) \Del \vt  \right\|_{L^q(\Omega)}^p \dt
\]
can be ``absorbed'' by their counterparts on the left--hand side of the inequalities \eqref{m7} and \eqref{m8} respectively. The same applies to 
the last integral in \eqref{b7a} as long as $T = T_{\rm eff}$ is small enough. These observation complete the desired bounds in terms of $\vr$ and $\vu$.

Finally, going back to {\eqref{m8} }we observe
\begin{align}
\int_0^T &\left\| \frac{1}{\vr} \mathbb{S}(\Ds \vu) : \Ds \vu \right\|^p_{L^q(\Omega)} \br	&\aleq 
\left( \delta + \sup_{\Omega} \frac{1}{\vr_0} \right) \int_0^T \| |\Grad \vu|^2 \|_{L^q(\Omega)}^p \dt = 
\left( \delta + \sup_{\Omega} \frac{1}{\vr_0} \right) \int_0^T \| \Grad \vu \|^{2p}_{L^{2q}(\Omega; R^3)} \dt,
\nonumber
	\end{align}
{for any $0< T \leq T_{\rm eff}$.}
Using the general version of Gagliardo--Nirenberg inequality (see Brezis, Mironescu \cite[Theorem 1]{BreMir}) we obtain 
\begin{equation} \label{est1}
\| \vu \|_{W^{s,q}(\Omega; R^3)} \leq \| \vu \|^{\frac{1}{2}}_{W^{\alpha,q}(\Omega; R^3)} \| \vu \|_{W^{2,q}(\Omega;R^3)}^{\frac{1}{2}}, 
\end{equation}
where
\[
s = 1 + \frac{\alpha}{2} ,\ \mbox{and}\ \alpha < 2 (1 - \frac{1}{p}) \ \Rightarrow \ B^{2(1 - \frac{1}{p})}_{q,p}(\Omega; R^3) 
\hookrightarrow B^{\alpha}_{q,q}(\Omega; R^3) \approx W^{\alpha,q}(\Omega; R^3).
\] 
Consequently, we have 
\begin{equation} \label{est}
\sup_{t \in [0, T_{\rm eff}] } \| \vu \|_{W^{\alpha,q}(\Omega)}  
\leq c(T_{\rm eff}) \Big( \| \vu \|_{L^p(0,T_{\rm eff}; W^{2,q}(\Omega;R^3))} +  \| \partial_t \vu \|_{L^p(0,T_{\rm eff}; L^{q}(\Omega;R^3))} \Big)	 
\end{equation}
Note carefully that the constant $c(T_{\rm eff})$ may blow up for $T_{\rm eff} \to 0$. Fortunately, $T_{\rm eff} > 0$ has already been fixed in the previous step.

Going back to \eqref{est1}, we have to show 
\[
W^{s,q}(\Omega;R^3) \hookrightarrow W^{1, 2q}(\Omega;R^3)\ \Leftrightarrow \ W^{\frac{\alpha}{2},q}(\Omega;R^3) \hookrightarrow L^{2q}(\Omega;R^3).
\]  
This is true if $\frac{\alpha}{2} q > 3$. If $\frac{\alpha}{2} q < 3$, we need 
\[
\frac{3q}{3 - \frac{\alpha}{2} q } > 2 q.
\]
As $\frac{\alpha}{2}$ can be chosen arbitrarily close to $(1 - \frac{1}{p})$, the desired conclusion follows from hypothesis \eqref{hhyp}.
 
The proof is now complete. Note we have used the fact that the constants in the maximal regularity estimates \eqref{MAX1}, \eqref{MAX2}
remain the same replacing $T$ by $T_{\rm eff} \leq T$.

\end{proof}

\section{Local existence of strong solutions}
\label{LE}

We are in a position to prove the local existence of strong solutions in both the $L^p-L^q$ framework claimed in Theorem \ref{Tmain}
and the mixed setting for the inhomogeneous Dirichlet boundary conditions. To the best of our knowledge, these results are new. Moreover, 
the method can be seen as an alternative proof of known results avoiding the explicit construction of approximate schemes. 

\subsection{Local existence $L^p-L^q$ setting}
\label{exipq}

Our goal is to show the existence of local in time solutions claimed in Theorem \ref{Tmain}.
The solutions are constructed via regular approximations in Valli's class. As already pointed out, 
this approach requires higher regularity of the underlying spatial domain.
We proceed in several steps. 	

\subsubsection{Data approximation}

We consider a sequence of initial data 
\[
(\vr^\ep_0, \vt^\ep_0, \vu^\ep_0) \in \mathcal{D}_{V,I} \cap \mathcal{D}_{p,q,I}, 
\]
along with the boundary data 
\[
(\vtB^\ep, \vu^\ep_B, q_B^\ep) \in \mathcal{D}_{V,B} \cap \mathcal{D}_{p,q,B},
\]
satisfying the compatibility conditions \eqref{L9}--\eqref{L12b}, 
along with a sequence of approximate potentials 
\[
G^\ep \in W^{3,2}(\Omega).
\]
In addition, we suppose 
\begin{align} 
	(\vr_0^\ep, \vt^\ep_0, \vu^\ep_0) &\to (\vr_0, \vt_0, \vu_0) \ \mbox{in} \ \mathcal{D}_{p,q,I}, \br  
	(\vtB^\ep, \vu^\ep_B, q_B^\ep)  &\to (\vtB, \vuB, q_B) \ \mbox{in} \ \mathcal{D}_{p,q,B}, \br
	\vr_0^\ep &\geq \underline{\vr} > 0, \vt^\ep_0 \geq \underline{\vt} > 0,\ \vt^\ep_B \geq \underline{\vt} > 0,\br 
	G^\ep &\to G \ \mbox{in}\ W^{1,q}(\Omega)
	\label{LE1}
\end{align}
as $\ep \to 0$.
Let $(\vre, \vte, \vue)_{\ep > 0}$ be the associated sequence of solutions 
to the NSF system in the space $\mathcal{S}_V[0,T]$, $0 < T < T^\ep_{{\rm max},V}$, 
the existence of which is guaranteed by Proposition \ref{PL1}. 

\subsubsection{Control functionals and hitting times}

We start with the following auxiliary result.

\begin{Lemma} \label{LLE1}
Suppose 
\begin{equation} \label{LE2a}
		3 < q < \infty, \ \frac{2q}{2q - 3} < p < \infty.
		\end{equation}
		
Then the embeddings 
\begin{equation} \label{LE2}
\mathcal{S}_{p,q}[0,T]  \hookrightarrow \hookrightarrow C([0,T] \times \Ov{\Omega}; R^4),
\end{equation}
and
\begin{align} 	 
L^p(0,T; W^{2,q}(\Omega)) \cap W^{1,p}(0,T; L^q(\Omega)) \hookrightarrow \hookrightarrow L^p (0,T; W^{1,\infty}(\Omega))
\label{LE3}
\end{align}	
are compact. In both cases, the embedding constant is independent of $T$.
	\end{Lemma}
	
\begin{proof}
First, we show 
\begin{equation} \label{es1}
S_{p,q}[0,T] \hookrightarrow C \Big([0,T]; W^{1,q}(\Omega) \times B^{2(1 - \frac{1}{p})}_{q,p}(\Omega) \times B^{2(1 - \frac{1}{p})}_{q,p}(\Omega;R^3) \Big),
\end{equation}
where the embedding constant is independent of $T$. As the embedding 
\[
W^{1,q}(\Omega) \times B^{2(1 - \frac{1}{p})}_{q,p}(\Omega) \times B^{2(1 - \frac{1}{p})}_{q,p}(\Omega;R^3) \hookrightarrow \hookrightarrow
C(\Ov{\Omega}; R^5)
\]
is compact, compactness claimed in \eqref{LE2} follows by a direct application of the Banach space version of Arzel\` a--Ascoli theorem.
	
First note that \eqref{es1} is obvious for the density $\vr$. 
As for $(\vt, \vu)$, we first write 
\[
(\vt, \vu) = (\vt - \tvt, \vu - \tvu) + (\tvt, \tvu), 
\]
where 
\[
(\tvu, \tvt) \in L^p(0,T; W^{2,q}(\Omega; R^4)) \cap 
W^{1,p}(0,T; L^q(\Omega; R^4)) 
\]
is en extension of the initial value, 
\[
(\tvt, \tvu)(0, \cdot) = (\vt, \vu)(0, \cdot). 
\]
Accordingly, we get the estimate 
\begin{equation} \label{as1}
\sup_{t \in (0,T)} \| (\vt - \tvt, \vu - \tvu) {(t, \cdot)} \|_{B^{2(1-\frac{1}{p})}_{q,p}(\Omega; R^4)}^p \aleq 
\int_0^T \left[ \| (\vt - \tvt, \vu - \tvu) \|^p_{W^{2,q}(\Omega; R^4)}+
\| \partial_t (\vt - \tvt, \vu - \tvu) \|^p_{W^{2,q}(\Omega; R^4)} \right] 
\end{equation}
As $(\vt - \tvt, \vu - \tvu)(0, \cdot) = 0$, the embedding constant is independent of $T$, cf. Denk \cite[Lemma 2.4]{Denk}.

Next, 
\[
\| (\vt, \vu)(0) \|_{B^{2(1 - \frac{1}{p})}_{q,p} }^p = \| (\tvt, \tvu)(0) \|_{B^{2(1 - \frac{1}{p})}_{q,p} }^p  = \inf \int_0^\infty \left[ \| \tvt, \tvu \|^p_{W^{2,q}(\Omega; R^4)}+
\| \partial_t (\tvt,\tvu) \|^p_{W^{2,q}(\Omega; R^4)} \right], 
\]
where the infimum is taken over all possible extensions $(\tvt, \tvu)$ of the initial data. Consequently, we may choose $(\tvt, \tvu)$ so that 
\[
\int_0^\infty \left[ \| \tvt, \tvu \|^p_{W^{2,q}(\Omega; R^4)}+
\| \partial_t (\tvt,\tvu) \|^p_{W^{2,q}(\Omega; R^4)} \right]  
\leq 2 \| (\vt, \vu)(0) \|_{B^{2(1 - \frac{1}{p})}_{q,p} }^p.
\]
Going back to \eqref{as1} we may infer that 
\begin{align} 
	\sup_{t \in (0,T)} &\| (\vt - \tvt, \vu - \tvu){ (t, \cdot)} \|_{B^{2(1-\frac{1}{p})}_{q,p}(\Omega; R^4)}^p \br&\aleq 
	\int_0^T \left[ \| (\vt, \vu) \|^p_{W^{2,q}(\Omega; R^4)}+
	\| \partial_t (\vt, \vu) \|^p_{W^{2,q}(\Omega; R^4)} \right] + 2 \| (\vt, \vu)(0) \|_{B^{2(1 - \frac{1}{p})}_{q,p} }^p. 
	\label{as2} 
\end{align}
Finally, we get 
\begin{align}
\sup_{t \in (0, T)}&\| (\tvt, \tvu)(t, \cdot) \|_{B^{2(1-\frac{1}{p})}_{q,p}(\Omega; R^4)}^p	\leq 
\sup_{t \in (0,\infty)}\| (\tvt, \tvu)(t, \cdot) \|_{B^{2(1-\frac{1}{p})}_{q,p}(\Omega; R^4)}^p\br  &\aleq \int_0^\infty \left[ \| \tvt, \tvu \|^p_{W^{2,q}(\Omega; R^4)}+
\| \partial_t (\tvt,\tvu) \|^p_{W^{2,q}(\Omega; R^4)} \right]   
\leq 2 \| (\vt, \vu)(0) \|_{B^{2(1 - \frac{1}{p})}_{q,p} }^p ,
\nonumber
\end{align}
which completes the proof of \eqref{es1}. Note that the norm introduced in \eqref{L3g} played a crucial role in the above arguments.

Finally, \eqref{LE3} follows from the compact embedding 
\[
W^{2,q}(\Omega; R^3) \hookrightarrow \hookrightarrow W^{1,\infty}(\Omega; R^3))
\]
combined with the Aubin--Lions compactness argument.
	
	\end{proof}	
	
Given the sequence of solutions $(\vre, \vte, \vue)_{\ep > 0}$, we define a sequence of \emph{control functionals} 
\[
\mathcal{F}_\ep (\tau) = \| (\vre, \vte, \vue) (\tau, \cdot) \|_{C(\Ov{\Omega}; R^5)} + \int_0^\tau \| (\vte, \vue) (t, \cdot) \|^p_{W^{1,\infty}(\Omega; R^4)} \dt
\]
for any $0 \leq \tau < T^\ep_{\rm max, V}$. 

Next, fix $M > 0$ large enough so that 
\begin{equation} \label{LE4}
	\mathcal{F}_\ep (0) =  \| (\vr^\ep_0, \vt^\ep_0, \vu^\ep_0)  \|_{C(\Ov{\Omega};R^5)} < \frac{M}{2} \ \mbox{uniformly for}\ \ep \to 0.
\end{equation}	
In accordance with Lemma \ref{LLE1}, $\mathcal{F}^\ep \in C[0,T]$ for any $T < T_{\rm max,V}^\ep$. 

Consider the ``hitting'' time 
\[
T_M^\ep = \sup_{ \tau \in [0, T_{\rm max})} \left\{  \mathcal{F}_\ep (\tau) < M  \right\}.
\]
By virtue of \eqref{LE4}, $T^\ep_M > 0$, and, in accordance with the conditional regularity criterion \eqref{c5}, we conclude, 
\[
0 < T_M^\ep < T^\ep_{{\rm max}, V}.
\]
Finally, we fix $T_{\rm eff} > 0$ as in \eqref{CR2}, specifically, 
\[
0 < T_{\rm eff} =  \Lambda^{-1} \left( \left\| (\vr_0, \vt_0, \vu_0) \right\|_{\mathcal{D}_{p,q,I}} + \left\| \frac{1}{\vr_0} \right\|_{C(\Ov{\Omega})} 
+ \| G \|_{W^{1,q}(\Omega)} + M \right).  
\]
In particular, $T_{\rm eff}$ is positive independent of $\ep$.

Set
\[
\overline{T}_\ep = \min\{ T^\ep_M; T_{\rm eff} \} > 0, 
\]
and 
\[
\Ov{\mathcal{F}}_\ep (\tau) = \mathcal{F}_\ep (\tau \wedge \Ov{T}_\ep),\ \tau \in [0, T_{\rm eff}].
\]

By virtue of the conditional regularity bound \eqref{CR1} and Lemma \ref{LLE1}, the sequence $( \Ov{\mathcal{F}}_\ep )_{\ep > 0}$ is precompact 
in $C[0,T_{\rm eff}]$, as it is precompact in $[0,T^\ep_M]$ independently of $\ep$.

Extracting a suitable subsequence we may assume 
\[
0 \leq \Ov{\mathcal{F}}_\ep \leq M,\ \Ov{\mathcal{F}}_\ep \to \mathcal{F} \ \mbox{in}\ C[0, T_{\rm eff}],\ 
\mathcal{F}(0) = \| (\vr_0, \vt_0, \vu_0)  \|_{C(\Ov{\Omega})} \leq \frac{M}{2}.
\]
Consequently, there exists $0< T \leq T_{\rm eff}$ such that 
\[
0 \leq \mathcal{F}(\tau) < M \ \mbox{for any}\ \tau \in [0,T] \ \Rightarrow \ 
0 \leq \Ov{\mathcal{F}}_\ep (\tau) < M \ \mbox{for any}\ \tau \in [0,T] 
\]
for all $\ep > 0$ small enough, meaning 
\begin{equation} \label{LE5} 
	\Ov{\mathcal{F}}_\ep(\tau) =  {\mathcal{F}}_\ep(\tau) = \| (\vre, \vte, \vue) (\tau, \cdot) \|_{C(\Ov{\Omega};R^5)} + \int_0^\tau \| (\vte ,\vue) (t, \cdot) \|^{ p}_{W^{1,\infty}(\Omega; R^4)} \dt < M \ \mbox{for any}\ \tau \in [0,T].
	\end{equation}
	
Finally, the estimate \eqref{LE5} along with the conditional regularity result \eqref{CR1} imply uniform bound,
\[
	\left\| (\vre, \vte, \vue) \right\|_{\mathcal{S}_{p,q} [0,T]} + \sup_{t \in [0, T]} \left\| \left(\frac{1}{\vre}, \frac{1}{\vte} \right) \right\|_{C(\Ov{\Omega})} 	\leq c(M) \ \mbox{uniformly for}\ \ep \to 0.
\]
Consequently, passing to a new subsequence as the case may be, we may suppose 
\begin{align} \
(\vte, \vue) &\to (\vt, \vu) \ \mbox{weakly in}\ L^p(0,T; W^{2,q}(\Omega; R^4)) \cap 
W^{1,p}(0,T; L^q(\Omega; R^4),\br \vre &\to \vr \ \mbox{weakly-(*) in}\ L^\infty(0,T; W^{1,q}(\Omega)).
\nonumber
\end{align}
In view of the strong convergence of the data postulated in \eqref{LE1}, and the compact embeddings stated in Lemma \ref{LLE1}, 
it is a routine matter to conclude that the triple $(\vr, \vt, \vu)$ is the strong solution of the NSF system in $[0,T]$ claimed in 
Theorem \ref{Tmain}. 

\subsection{Local existence in the mixed setting}

	\begin{Theorem}[{\bf Local existence - mixed framework}] \label{TEL2}
		
		Let 
		\[ 
			3 < q \leq 6
		\]
		be given.
		Let $\Omega \subset R^3$ be a bounded domain of class $C^4$, and let $G = G(x)$, $G \in W^{2,2}(\Omega)$ be a given potential. Let the initial data $(\vr_0, \vt_0, \vu_0)$ satisfy \eqref{MS5}, and let the boundary data $(\vu_B, \vt_B)$ belong to the class  \eqref{MS6}, \eqref{MS7}, respectively, 
		where 
		\[
		\min_{\Omega} \vr_0 \geq \underline{\vr} > 0,\ 
		\min_{\Omega} \vt_0 \geq \underline{\vt} > 0,\ \min_{\Gamma_D} \vtB \geq \underline{\vt} > 0,\  q_B \geq 0.
		\]
		In addition, let the compatibility conditions 
		\begin{align} 
			\vu_0|_{\partial \Omega} &= \vuB, \br
			\vt_0|_{\Gamma_D} &= \vtB,\ 
			\Grad \vt_0 \cdot \vc{n}|_{\Gamma_N} = q_B 
			\nonumber
		\end{align}
		hold. Finally, we suppose either $\Gamma_D \ne \emptyset$ or $q_B = 0$.

		Then there exists $T > 0$ such that the NSF system \eqref{p1}--\eqref{p6}, 
		with the initial conditions \eqref{p7}, and the boundary conditions \eqref{p8} 
		admits a unique solution $(\vr, \vt, \vu)$ in the class \eqref{MS1}--\eqref{MS3}.  
		
	\end{Theorem}	

We omit the proof as it can be carried over in the same way as in the preceding section, with Proposition \ref{PCR1} in place of 
Proposition \ref{PCR2}. 

\begin{Remark} \label{ERER}
The conditional regularity estimates \eqref{c5a} established originally for smoother solutions extend by compactness/density arguments 
to the less regular solutions obtained in Theorem \ref{TEL2}.
\end{Remark}	
	
\section{Conditional regularity in the $L^p - L^q$ framework}
\label{R}

Our ultimate goal is to establish conditional regularity estimates in the $L^p-L^q$ framework and thus complete the proof of Theorem \ref{Tmain}.

	\begin{Theorem}[{\bf Conditional regularity - $L^p-L^q$ framework}] 
		\label{TR1}
		Let the hypotheses of Theorem \ref{Tmain} be satisfied.  
		
		Then there exists a non--decreasing (positive) function 
		$\Lambda : [0, \infty) \to (0, \infty)$ such that the solution $(\vr, \vt, \vu)$ admits the estimate 
				\begin{align} 
			&\left\| (\vr, \vt, \vu) \right\|_{\mathcal{S}_{p,q} [0,T]} + \sup_{t \in [0, T]} \left\| \left(\frac{1}{\vr}, \frac{1}{\vt} \right) {(t, \cdot)}\right\|_{C(\Ov{\Omega}; {R^2})} 
			\br &\leq 
			\Lambda \left(\left\| (\vr_0, \vt_0, \vu_0) \right\|_{\mathcal{D}_{p,q,I}} + \left\| (\vuB, \vtB, q_B) \right\|_{\mathcal{D}_{p,q,B}}  
			+ \left\| \left(\frac{1}{\vr_0}, \frac{1}{\vt_0} \right) \right\|_{C(\Ov{\Omega}; {R^2})} + \| G \|_{W^{1,q}(\Omega)} \right. \br 
			&+ \left.	  \left\| (\vr, \vt, \vu ) \right\|_{C([0,T] \times \Ov{\Omega}; R^5)}	\right).
			\label{R1}
		\end{align}
		for any $0 \leq T < T_{{\rm max}, p,q}$.	
	\end{Theorem}	
	
Similarly to the above, $T_{{\rm max}, p,q}$ denotes the life span of the strong solution constructed in Section \ref{exipq}. As we shall see below, it is actually independent of the specific value of the exponents $p,q$.

The rest of this section is devoted to the proof of Theorem \ref{TR1}.	

\subsection{Estimates for more regular data}

We start with the case of more regular data, specifically, 
\[
(\vr_0, \vt_0, \vu_0) \in \mathcal{D}_{p,q,I} \cap \mathcal{D}_{ChK, I},\ 
(\vuB, \vtB, q_B) \in \mathcal{D}_{p,q,B} = \mathcal{D}_{ChK, B}, \ G \in W^{2,q}(\Omega).
\]
Accordingly, the solution satisfies the bound \eqref{c5a} from Proposition \ref{PCR1}, cf. also Remark \ref{ERER}:
		\begin{align} 
	\left\| (\vr, \vt, \vu) \right\|_{\mathcal{S}_{ChK} [0,T]} &+ \sup_{t \in [0,T]} \left\| \left(\frac{1}{\vr}, \frac{1}{\vt} \right) {(t, \cdot)}\right\|_{C(\Ov{\Omega}; {R^2})} 
	\br &\leq 
	\Lambda \left(T + \left\| (\vr_0, \vt_0, \vu_0) \right\|_{\mathcal{D}_{ChK,I}} + \left\| (\vuB, \vtB, q_B) \right\|_{\mathcal{D}_{ChK,B}}  
	+ \left\| \left(\frac{1}{\vr_0}, \frac{1}{\vt_0} \right) \right\|_{C(\Ov{\Omega}; {R^2})} \right. \br  &+ \left.  \left\| (\vr, \vt, \vu ) \right\|_{L^\infty((0,T) \times \Ov{\Omega}; R^5)} + \left\| G \right\|_{W^{2,q}(\Omega)}	\right),\ 3 < q \leq 6
	\label{R3}
\end{align}	
for any $0 < T < T_{{\rm max}, ChK}$.

Now, revisiting the maximal regularity estimates \eqref{MAX1}, \eqref{MAX2} we observe that 	 
\begin{align} 
\| (\vr, \vt, \vu) \|_{S_{p,q}[0,T]} &\leq \Lambda \left( \left\| (\vr, \vt, \vu) \right\|_{\mathcal{S}_{ChK} [0,T]} + \sup_{t \in [0,T]} \left\| \left(\frac{1}{\vr}, \frac{1}{\vt} \right) {(t, \cdot)}\right\|_{C(\Ov{\Omega}; {R^2})}   \right. \br 
&+ \left. T+ \left\| (\vr_0, \vt_0, \vu_0) \right\|_{\mathcal{D}_{p,q,I}} + \left\| (\vuB, \vtB, q_B) \right\|_{\mathcal{D}_{p,q,B}} + \left\| G \right\|_{W^{2,q}(\Omega)}   \right).
\label{R4} 
\end{align}	
Unlike \eqref{R3} that holds for a restricted range of $q$ amenable to the mixed class, the estimate 
\eqref{R4} can be extended to arbitrary $p,q$ satisfying \eqref{LEa1} by an induction argument. 

Thus the conclusion of Theorem \ref{TR1} holds for more regular initial data, specifically:  
		\begin{align} 
	&\left\| (\vr, \vt, \vu) \right\|_{\mathcal{S}_{p,q} [0,T]} + \sup_{t \in [0, T]} \left\| \left(\frac{1}{\vr}, \frac{1}{\vt} \right) {(t, \cdot)} \right\|_{C(\Ov{\Omega}; {R^2})} 
	\br &\leq 
	\Lambda \left(\left\| (\vr_0, \vt_0, \vu_0) \right\|_{\mathcal{D}_{p,q,I} \cap \mathcal{D}_{ChK,I}} + \left\| (\vuB, \vtB, q_B) \right\|_{\mathcal{D}_{p,q,B}}  
	+ \left\| \left(\frac{1}{\vr_0}, \frac{1}{\vt_0} \right) \right\|_{C(\Ov{\Omega}; { R^2})} \right. \br &+  \left. \left\| (\vr, \vt, \vu ) \right\|_{L^\infty((0,T) \times \Ov{\Omega}; R^5)} + \| G \|_{W^{2,q}(\Omega)}	\right).
	\label{R5}
	\end{align}

\subsection{Initial smoothing}	

In view of the bound \eqref{R5}, the proof of Theorem \ref{TR1} is complete as soon as we observe that solutions emanating from the data in $\mathcal{D}_{p,q, I}$ will immediately become more regular, namely 
\begin{equation} \label{pr16}
(\vr (\tau, \cdot), \vt(\tau, \cdot), \vu(\tau, \cdot)) \in \mathcal{D}_{p,q,I} \cap \mathcal{D}_{ChK,I} 
\end{equation}
for any $\tau > 0$. This is obviously the case for the density, where the two classes coincide, while for $\vu$, $\vt$, we have
\[
(\vt, \vu) \in L^p (0,T; W^{2,q}(\Omega; R^4)).
\]
As $q > 3$, the relation \eqref{pr16} holds for a.a. $\tau \in (0,T)$. Moreover, by the mean value theorem, 
\[
\| (\vt, \vu)(\xi, \cdot) \|_{W^{2,q}(\Omega; R^4)}^p = \frac{1}{T} \| (\vt, \vu) \|_{L^p(0,T; W^{2,q}(\Omega; R^4))}^p \ \mbox{for some}\ 
\xi \in [0,T].
\]

Consequently, the the proof of Theorem \ref{Tmain} is complete as soon as we establish the following result. 

\begin{Lemma} \label{LR1}
Consider a bounded set of initial/boundary data 
\[
\mathcal{B} = \left\{ (\vr_0, \vt_0, \vu_0, \vt_B, \vu_B, q_B) \ \Big| \| (\vr_0, \vt_0, \vu_0) \|_{\mathcal{D}_{p,q,I}} + 
\|( \vt_B, \vu_B, q_B) \|_{\mathcal{D}_{p,q,B}} \leq K
\right\}.
\]

Then there exists $T(K) > 0$ such that the local solutions $(\vr, \vt, \vu)$ corresponding to the data in $\mathcal{B}$ exist on $[0,T]$, and 
\[
\| (\vt, \vu) \|_{L^p(0,T; W^{2,q}(\Omega; R^4))}^p < \Lambda (K).
\]

	\end{Lemma}
	
\begin{proof}
	
The arguments are similar to the existence part of the proof of Theorem \ref{Tmain}.  	
Given the data 
\[
(\vr_0, \vt_0, \vu_0, \vt_B, \vu_B, q_B) \in \mathcal{B}, 
\]
we set 
\[
\mathcal{F}[ \vr_0, \vt_0, \vu_0, \vt_B, \vu_B, q_B] (\tau) = \| (\vr, \vt, \vu)(\tau, \cdot) \|_{C(\Ov{\Omega};R^5)} + 
\int_0^\tau \| (\vt, \vu) (t, \cdot) \|^{{p}}_{W^{1,\infty}(\Omega; R^4)} \dt,  
\]
where $(\vr, \vt, \vu)$ is the corresponding strong solution. 	

Now, we fix $T_{\rm eff}$ as in \eqref{CR2}, 
\[
T_{\rm eff} = \Lambda^{-1}(K + M) > 0, 
\]
where $K$ is the ``radius'' of the data set and $M = M(K) > 0$ is a given number chosen in such a way that
\begin{equation} \label{R6}
\sup \left\{ 	\mathcal{F}[ \vr_0, \vt_0, \vu_0; \vt_B, \vu_B, q_B] (0) \Big| \ (\vr_0, \vt_0, \vu_0, \vt_B, \vu_B, q_B) \in \mathcal{B} \right\} \leq 
\frac{M}{2}.
	\end{equation}
	
Next, we define 
\[
T_M = 
T_M [ \vr_0, \vt_0, \vu_0; \vt_B, \vu_B, q_B] = \sup \left\{ 0 \leq \tau \leq T_{\rm eff} \Big|\ \mathcal{F}[ \vr_0, \vt_0, \vu_0; \vt_B, \vu_B, q_B](\tau) < M \right\}.
\]	
Note carefully that $T_M > 0$ because of \eqref{R6}.
Finally, we set 
\[
T = T(K) = \inf \left\{ T_M [ \vr_0, \vt_0, \vu_0; \vt_B, \vu_B, q_B] \ \Big| \ (\vr_0, \vt_0, \vu_0; \vt_B, \vu_B, q_B) \in \mathcal{B} \right\}. 
\]
In view of the conditional regularity estimates \eqref{CR1} established in Proposition \ref{PCR2}, the time $T$ meets the conclusion of 
Lemma \ref{LR1} as soon as we show $T > 0$.

Arguing by contradiction, we construct a sequence of data 
\[
( \vr^\ep_0, \vt^\ep_0, \vu^\ep_0; \vt^\ep_B, \vu^\ep_B, q^\ep_B) \in \mathcal{B} 
\]
such that
\begin{equation} \label{R7}
T_M^\ep =	T_M [\vr^\ep_0, \vt^\ep_0, \vu^\ep_0; \vt^\ep_B, \vu^\ep_B, q^\ep_B] \to 0 \ \mbox{as} \ \ep \to 0. 
\end{equation}
We consider the associated sequence 
\[
\mathcal{F}_\ep (\tau) = \mathcal{F}[ \vr^\ep_0, \vt^\ep_0, \vu^\ep_0, \vt^\ep_B, \vu^\ep_B, q^\ep_B](\tau),  
\]
with 
\[
\Ov{\mathcal{F}}_\ep (\tau) = \mathcal{F}_\ep (\tau \wedge T^\ep_M). 
\]
Now, similarly to the existence part of the proof of Theorem \ref{Tmain}, we use Lemma \ref{LLE1} to conclude 
\[
\Ov{\mathcal{F}}_\ep \to \mathcal{F} \in C[0, T_{\rm eff}], 
\]
where 
\[
0 \leq \mathcal{F}(\tau) < M \ \mbox{on some interval} \ [0, \alpha] ,\ \alpha > 0.
\]
This implies 
\[
0 \leq \Ov{\mathcal{F}}_\ep(\tau) = \mathcal{F}_\ep (\tau) < M \ \mbox{in} \ [0, \alpha], 
\]
meaning $T_M^\ep > \alpha > 0$ in contrast with \eqref{R7}.

	\end{proof}	
	
\centerline{\bf Acknowledgement}

	This work was done while the authors were participating in the Oberwolfach 
	Research Fellows Program at Mathematisches Forschungsinstitut Oberwolfach. They warmly
	thank the Institute for its kind hospitality and the excellent research environment it provided.
	
	\def\cprime{$'$} \def\ocirc#1{\ifmmode\setbox0=\hbox{$#1$}\dimen0=\ht0
		\advance\dimen0 by1pt\rlap{\hbox to\wd0{\hss\raise\dimen0
				\hbox{\hskip.2em$\scriptscriptstyle\circ$}\hss}}#1\else {\accent"17 #1}\fi}

%\bibliography{citace}
%\bibliographystyle{plain}

\end{document}